\documentclass{article}
\usepackage[utf8]{inputenc}

\title{On the rate of convergence in quenched Voronoi percolation\thanks{Research in part supported by the Swedish Research Council grant 2016-04442 (DA), FAPERJ bolsa Nota 10 proc. E-26/200.194/2020 and Coordenação de Aperfeicoamento de Pessoal de Nível Superior - Brasil (CAPES) - finance code 001 (DdlR), CNPq bolsa de produtividade em pesquisa proc. 307521/2019-2 and FAPERJ Jovem cientista do nosso estado proc. 202.713/2018 (SG).}}
\author{Daniel Ahlberg, Daniel de la Riva and Simon Griffiths}
\date{\vspace{-5ex}}
\usepackage{amsmath}
\usepackage{color}
\usepackage{float}
\usepackage{scalerel}[2016/12/29]
\usepackage{amsthm}
 \usepackage{bbm}
\usepackage{dirtytalk}

\usepackage{enumerate}
\usepackage[maxbibnames=99]{biblatex}
\usepackage{amsfonts}
\usepackage{graphicx}
\usepackage{sectsty}
\usepackage{dsfont}
\usepackage{amsmath,amssymb,amsfonts}
\usepackage{tikz}
\usetikzlibrary{hobby}
\sectionfont{\large}
\graphicspath{ {images/} }
\newtheorem{theorem}{Theorem}

\newtheorem{proposition}[theorem]{Proposition}
\newtheorem{definition}[theorem]{Definition}
\newtheorem{remark}[theorem]{Remark}
\newtheorem{lemma}[theorem]{Lemma}
\newtheorem{corollary}[theorem]{Corollary}
\newtheorem{claim}[theorem]{Claim}
\def\E{{\mathbb{E}}}
 \def\F{\mathcal{F}}
\DeclareMathOperator{\V}{Var} 
\def\P{{\mathbb{P}}}
\def\Var{\textup{Var}}
\def\Inf{\textup{Inf}}
\def\eps{\varepsilon}
\numberwithin{theorem}{section}
\usepackage{biblatex}
\def\R{{\mathbb{R}}}
\def\Z{{\mathbb{Z}}}
\def\N{{\mathbb{N}}}
\def\Q{{\mathbb{Q}}}

\def\A{{\mathcal{A}}}
\def\Ind{{\mathbbm{1}}}

\def\Po{{\mathbb P}^\ast}
\DeclareFieldFormat{pages}{#1}
\renewbibmacro{in:}{%
  \ifentrytype{article}
    {}
    {\bibstring{in}%
     \printunit{\intitlepunct}}}
\newcommand\norm[1]{\left\lVert#1\right\rVert}
\newcommand{\bn}[1]{Bin(n,#1)}

\newcommand{\poh}[1]{\P_{Po|_{\mathbb{H}}}(#1)}
\newcommand{\poq}[1]{\P_{Po|_{\Q}}(#1)}

\begin{document}

\maketitle

\begin{abstract}
Position $n$ points uniformly at random in the unit square $S$, and consider the Voronoi tessellation of $S$ corresponding to the set $\eta$ of points. Toss a fair coin for each cell in the tessellation to determine whether to colour the cell red or blue. Let $H_S$ denote the event that there exists a red horizontal crossing of $S$ in the resulting colouring.
In 1999, Benjamini, Kalai and Schramm conjectured that knowing the tessellation, but not the colouring, asymptotically gives no information as to whether the event $H_S$ will occur or not. More precisely, since $H_S$ occurs with probability $1/2$, by symmetry, they conjectured that the conditional probabilities $\P(H_S|\eta)$ converge in probability to 1/2, as $n\to\infty$. This conjecture was settled in 2016 by Ahlberg, Griffiths, Morris and Tassion. In this paper we derive a stronger bound on the rate at which $\P(H_S|\eta)$ approaches its mean. As a consequence we strengthen the convergence in probability to almost sure convergence.

\end{abstract}

\section{Introduction} 

In a seminal paper from 1999, Benjamini, Kalai and Schramm~\cite{BKS99} introduced the concept of noise sensitivity for Boolean functions. The study of sensitivity and stability of Boolean functions rapidly grew into a new area of research, and two books cover much of the early development~\cite{garban_steif_book,odonnell_2014}. Benjamini, Kalai and Schramm~\cite{BKS99} outlined methods for the study of noise sensitivity, connecting noise sensitivity to influences of bits and revealment of algorithms, and these methods remain central to this day. Using these methods, the authors gave examples of noise sensitive Boolean functions, most significantly functions encoding crossing events in Bernoulli percolation on $\Z^2$.

In~\cite{BKS99}, a set of conjectures was presented that has been leading much of the development since. One of these conjectures concerned stronger quantitative bounds on the noise sensitivity of crossings in Bernoulli percolation, which would have implications for the study of so-called `exceptional times' in dynamical percolation. Such bounds were first obtained by Schramm and Steif~\cite{SS10}, then by Garban, Pete and Schramm~\cite{GPS10} and more recently Tassion and Vanneuville~\cite{TV}. In another direction, the study of noise sensitivity was extended to percolation models in the continuum by Ahlberg, Broman, Griffiths and Morris~\cite{ABGM14}, Ahlberg and Baldasso~\cite{AB18}, Vanneuville~\cite{vanneuville3}, and most recently Last, Peccati and Yogeshwaran~\cite{LPY}.
One of the conjectures from~\cite{BKS99} concerned such a continuum model, known as Voronoi percolation, and that conjecture has motivated the current work.

Let $S:=[-\frac12,\frac12]^2$ be the unit square. Position $n$ points independently and uniformly at random in $S$. Let $\eta$ denote the resulting set of positions. The \emph{Voronoi tessellation} of $S$ with respect to $\eta$ is the division $(V(u))_{u\in\eta}$ of $S$, where the cell $V(u)$ consists of all point of $S$ closer to $u$ than to any other point in $\eta$, i.e.\
 $$
 V(u):= \big\{ x\in S: \norm{u-x}_{2}\leq\norm{v-x}_{2} \  \text{for every} \ v\in\eta\big\},
 $$
where $\|\cdot\|_2$ refers to Euclidean distance. Next, colour the cells of the tessellation `red' or `blue' with equal probability, independently of one another; denote by $\P_S$ the resulting measure. We denote by $H_{S}$ the event that there exists a red horizontal crossing in the resulting colouring of $S$, i.e., a continuous path connecting the left-hand side of $S$ to the right-hand side and contained in the union of red cells. 

A standard duality argument, due to the symmetry between red and blue, shows that $\P_S(H_{S})=1/2$. Benjamini, Kalai and Schramm~\cite{BKS99} conjectured that knowing the Voronoi tessellation, but not knowing the colouring of the cells, typically gives almost no information of whether the colouring contains a red left-right crossing. More precisely, they asked whether for every $\eps>0$ we have for all large $n$ that
$$\P_S\big(|\P(H_{S}|\eta)-1/2 |>\eps \big)<\eps. $$
Ahlberg, Griffiths, Morris and Tassion~\cite{QuenchedVoronoi} confirmed this conjecture. More precisely, they showed that there exists $c>0$ such that for all large $n$ we have
\begin{equation}\label{eq:AGMT}
\P_S\big( |\P(H_{S}|\eta)-1/2|> n^{-c}\big)< n^{-c}.
\end{equation}


The approach of~\cite{QuenchedVoronoi} would in fact still only give a bound of the form $n^{-c}$ for the probability of a fixed deviation of $\P(H_S|\eta)$ away from its mean. In this paper we prove a stronger bound for the probability of a large deviation of this form, via estimates of its moment generating function.

\begin{theorem}\label{Improved}
 There exists $N\ge1$ such that for all $t>0$ and $n\ge N$ we have
$$\P_S\big( |\P(H_{S}|\eta)-1/2| \geq t \big) \leq 4\exp\big(-t\,e^{(\log\log n)^2}\big).$$
\end{theorem}

As a straightforward corollary of the above theorem, which does not follow from the work in~\cite{QuenchedVoronoi}, we obtain the following almost sure statement via the Borel-Cantelli Lemma.

\begin{corollary}
Suppose that we continue indefinitely to position points uniformly at random in $S$. Then,
$$ \P_{S} \Big(\lim_{n \to \infty }\P(H_{S}|\eta) = 1/2 \Big) =1. $$ 
\end{corollary}

Our method of proof is based on the same ideas that underlie the approach of~\cite{QuenchedVoronoi}. However, the only results from~\cite{QuenchedVoronoi} that we shall need are those in Sections~3.1 and~3.2 of that paper. In order to outline our method of proof, and to compare it to the approach taken in~\cite{QuenchedVoronoi}, we shall proceed with a brief outline of noise sensitivity of Boolean functions.


\subsection{Noise sensitivity of Boolean functions}

A $\{0,1\}$-valued function on the discrete cube $\{0,1\}^n$ is known as a Boolean function. There is a one-to-one correspondence between Boolean functions $f:\{0,1\}^n\to\{0,1\}$ and subsets of $\{0,1\}^n$ via the mapping $f\mapsto\{\omega:f(\omega)=1\}$, and we shall henceforth work with these subsets as it suits our purposes. We let $\P$ refer to uniform measure on the discrete cube $\{0,1\}^n$, and $\E$ expectation with respect to $\P$. (Note that this is consistent with our convention to drop the subscript of $\P_S$ when conditioning on $\eta$, as the conditional measure corresponds to a uniform two-colouring of the tessellation.)

Given $\omega\in\{0,1\}^n$ we obtain a perturbation $\omega^\eps$ of $\omega$ by independently re-randomizing each bit with probability $\eps>0$. A sequence $(A_n)_{n\ge1}$ of events $A_n\subseteq\{0,1\}^n$, is said to be \emph{noise sensitive} if for every $\eps>0$, as $n\to\infty$,
$$
\E\big[\Ind_{A_n}(\omega)\Ind_{A_n}(\omega^\eps)\big]-\E\big[\Ind_{A_n}(\omega)\big]^2\to0.
$$
One of the main achievements in~\cite{BKS99} was the characterization of noise sensitivity for a sequence of Boolean functions in terms of their influences. Given an event $A\subseteq\{0,1\}^n$ we define the \emph{influence} of $A$ with respect to the bits $j=1,2,\ldots,n$ as $\Inf_j(A):=\P(\Ind_{A}(\omega)\neq\Ind_{A}(\sigma_j\omega))$, where $\sigma_j$ is the operator that replaces the $j$th bit $\omega_j$ of $\omega$ by $1-\omega_j$. A central result in~\cite{BKS99}, which has come to be referred to as the BKS Theorem, states that a sufficient condition for a sequence of events $(A_n)_{n\ge1}$ (or the sequence of indicators corresponding to these events) to be noise sensitive is that
\begin{equation}\label{eq:BKScond}
\lim_{n\to\infty}\sum_{j=1}^n\Inf_j(A_n)^2=0.
\end{equation}

Also in~\cite{BKS99}, the authors devised a method based on algorithmic exploration of the bits of $\omega$, as a means to verify condition~\eqref{eq:BKScond} in specific examples. This `algorithm method' was later refined by Schramm and Steif~\cite{SS10}. A \emph{randomized algorithm} is an algorithm that sequentially queries bits of $\omega$, and where the next bit queried is determined by some probability measure depending on the information obtained thus far. An algorithm $\mathcal{A}$ is said to \emph{determine} an event $A$ if, when the algorithm ends, it has determined whether $\omega\in A$ or not. The \emph{revealment} of $\mathcal{A}$ is defined as $\delta(\mathcal{A}):=\max_j\P(\mathcal{A}\text{ queries bit }j)$. The Revealment Theorem, due to Schramm and Steif~\cite{SS10}, implies the following:\footnote{Below, in the appendix, we offer a direct probabilistic proof of this version of the Revealment Theorem.} For any monotone event $A\subseteq\{0,1\}^n$ and randomized algorithm $\mathcal{A}$ that determines $A$ we have
\begin{equation}\label{eq:SS}
    \sum_{j=1}^n\Inf_j(A)^2\le\delta(\mathcal{A}).
\end{equation}
Combining~\eqref{eq:BKScond} and~\eqref{eq:SS} outlines the essence of the algorithm method from~\cite{BKS99}: identification of an algorithm with low revealment that determines the event in question.

\subsection{Outline of the argument}

The approach taken in this paper is greatly inspired by that of~\cite{QuenchedVoronoi}. One of the key steps in the proof of~\eqref{eq:AGMT} was an Efron-Stein-like inequality, linking the variance of $\P(H_S|\eta)$ to the conditional influences of $H_S$ given $\eta$:
\begin{equation}\label{eq:efronstein}
\Var_S\big(\P(H_S|\eta)\big)\le\E_S\bigg[\sum_{j\in\eta}\Inf_j(H_S|\eta)^2\bigg],
\end{equation}
where $\Inf_j(H_S|\eta):=\P\big(\Ind_{H_S}(\omega)\neq\Ind_{H_S}(\sigma_j\omega)|\eta\big)$. Since the conditional measure $\P(\,\cdot\,|\eta)$ coincides with uniform measure on the discrete cube,~\eqref{eq:efronstein} together with Chebyshev's inequality, allowed the authors of~\cite{QuenchedVoronoi} to respond to the conjecture from~\cite{BKS99} by describing an algorithm with low revealment. Interestingly, it turns out that in order to do so, a key ingredient is to show that, in the limit, the random variables $\P(H_S|\eta)$ do not accumulate mass at $0$ or $1$. That is, in order to prove that $\P(H_S|\eta)$ approaches $1/2$, a key first step consists (roughly speaking, we will elaborate on this below) of proving that for every $\eps>0$ there exists $\delta>0$ such that for all $n\ge1$
\begin{equation}\label{eq:nondeg}
\P_S\big(\P(H_S|\eta)\in(\delta,1-\delta)\big)>1-\eps.
\end{equation}

Once~\eqref{eq:AGMT} has been obtained, it is tempting to try to repeat the argument, replacing~\eqref{eq:nondeg} by the sharper bound in~\eqref{eq:AGMT}, in order to obtain a stronger result. However, since the sum of influences squared, with high probability, will decay at the inverse rate of a low-degree polynomial, the inequality in~\eqref{eq:efronstein} is unable to improve upon~\eqref{eq:AGMT}. To improve upon~\eqref{eq:AGMT}, we will need a replacement for~\eqref{eq:efronstein}.

Our proof of Theorem~\ref{Improved} will consist of two main components. The first is a bound on the sum of influences squared by means of an algorithm with low revealment. This step is similar to the analogous step in~\cite{QuenchedVoronoi}, but more carefully quantified to serve our needs.
The second component is an exponential inequality, which substitutes~\eqref{eq:efronstein}, and is derived via an estimate on the moment generating function of $\P(H_S|\eta)$.
The bound in Theorem~\ref{Improved} will be obtained by iterated use of the above two steps. Already the first iteration yields an improvement to~\eqref{eq:AGMT}, with a bound decaying at the inverse rate of an arbitrarily high-degree polynomial, and an induction argument strengthens the bound further.

There are two complications arising from the above iterative scheme. One of these relates to the fact that the expression in~\eqref{eq:nondeg} is not quite the input required to obtain a bound on the revealment of a suitable algorithm. Instead we require the analogous statement for $\P(H_R|\eta)$, where $R\subseteq S$ is a rectangle of aspect ratio $3:1$. That means that we will need to work in a greater generality in order to produce an output that can feed into the next stage of the iteration. That is, in order to prove Theorem~\ref{Improved}, we shall in fact be required to prove a more general statement.

Given a rectangle $R\subseteq\R^2$ we let $\P_R$ denote the probability measure corresponding to the positioning $n$ points uniformly at random in $R$, independently from one another, and the subsequent uniform two-colouring the Voronoi cells with respect to the set of points $\eta$. (The dependence on $n$ is suppressed from the notation as we believe that no confusion will arise.) In particular, for $\rho>0$ and $n\ge1$, we let
$$
R(\rho,n):=\Big[-\frac12\sqrt{\rho n},\frac12\sqrt{\rho n}\Big]\times\Big[-\frac12\sqrt{n/\rho},\frac12\sqrt{n/\rho}\Big]
$$
denote the rectangle with aspect ratio $\rho$ and area $n$ centered at the origin, and let $R(\rho):=R(\rho,1)$.

We shall prove the following theorem.

\begin{theorem}\label{thm:optimized}
For every $\theta>0$ there exists $N=N(\theta)$ such that for every $\rho\in[\theta,1/\theta]$, $n\ge N(\theta)$ and $t>0$ we have for any (axis parallel) rectangle $R'\subseteq R(\rho)$ of area at least $\theta$ that
$$
\P_{R(\rho)}\big(|\P(H_{R'}|\eta)-\P_{R(\rho)}(H_{R'})|\ge t\big)\le 4\exp\big(-t\,e^{(\log\log n)^2}\big).
$$
\end{theorem}

Note that Theorem~\ref{Improved} is obtained from Theorem~\ref{thm:optimized} for $\rho=1$ and $R'=R(1)=S$. We remark that a similar (but weaker) extension from squares to rectangles was obtained also in~\cite{QuenchedVoronoi}.

The other complication of the iterative scheme is that the two pillars of the argument are most easily derived in somewhat different settings. The step in which we bound the sum of influences squared by means of a low-revealment algorithm is most conveniently carried out when $\eta$ is obtained as the positions of a Poisson point process on $\R^2$, due to the spatial independence between disjoint regions. The step in which the bound on the sum of influences squared is transformed into a bound on the deviations of $\P(H_R|\eta)$ from its mean will, on the other hand, be carried out for $\eta$ consisting of $n$ independent points. That is, we shall be working with two different, but closely related, models. Alternatively, which will facilitate comparison between the two models, it will be convenient to think of them as two probability measures associated to the same measurable space.\footnote{We have not been explicit what this measurable space is, but one can think of it as the space of locally finite subsets of $\R^2$, or more formally as the space of locally finite counting measures on $\R^2$, equipped with a suitable $\sigma$-algebra; see e.g.~\cite{laspen18} for details.}

\subsection{Poisson Voronoi percolation}

In Poisson Voronoi percolation, $\R^2$ is partitioned into an `occupied' and `vacant' set based on a two-colouring of the Voronoi tessellation of a unit rate Poisson point process.

Henceforth, we shall refer to the measure $\P_R$, in which $\eta$ consists of $n$ independently chosen uniform positions in $R$, as the \emph{binomial model}. Moreover, we let $\Po$ denote the measure associated to a unit rate Poisson point process $\eta$ on $\R^2$, and the subsequent `uniform' two-colouring (more precisely, the cells are equally likely to be either red or blue, independently of one another) of the Voronoi tessellation with respect to $\eta$, and refer to this as the \emph{Poisson model}. The two measures $\P_R$ and $\Po$, on the measurable space we can denote by $(\Omega,\mathcal{F})$, are supported on finite and infinite configurations $\eta$, respectively. In either case, and for any rectangle $R'$, the event $H_{R'}$ refers to the existence of a red left-right crossing of $R'$ in the Voronoi tessellation of $\R^2$ with respect to $\eta$.

A fundamental theorem of Bollob\' as and Riordan~\cite{bolrio06} shows that the uniform colouring is `critical' for Poisson Voronoi percolation. Clearly, the existence of a left-right crossing of a square has probability $1/2$ for uniform colourings, and in~\cite{bolrio06} the authors showed that, for non-uniform colourings, this probability tends rapidly to either zero or one as the side length of the square increases. Building on their work, Tassion~\cite{TassionRSW} established a `box-crossing property', showing that for every $\theta>0$ there exists $c=c(\theta)>0$ such that for all $\rho\in[\theta,1/\theta]$ and $n\ge1$
\begin{equation}\label{eq:rsw_intro}
c<\Po\big(H_{R(\rho,n)}\big)<1-c.
\end{equation}

We shall in this paper need to compare the binomial model to the Poisson models. The need for comparison between the two models arose already in~\cite{QuenchedVoronoi}, but we shall here need to be more careful as we shall need to pass back and forth between them repeatedly.
As a consequence of these comparison lemmas, which are presented in Section~\ref{Comparison}, we will obtain from Theorem~\ref{thm:optimized} an analogous result for Poisson Voronoi percolation.

\begin{theorem}\label{thm:poptimized}
For every $\theta>0$ there exists $N=N(\theta)$ such that for every $\rho\in [\theta,1/\theta]$, $n\ge N(\theta)$ and $t>0$ we have that
$$
\Po\big(|\P(H_{R(\rho,n)}|\eta)-\Po(H_{R(\rho,n)})|\ge t\big)\,\le\,  6\exp\big(-t\,e^{(\log\log n)^2}\big).
$$
\end{theorem}

Together with~\eqref{eq:rsw_intro} it follows, in particular, that for $\rho\in[\theta,1/\theta]$ and all large $n$ we have
$$
\Po\big(c/2<\P(H_{R(\rho,n)}|\eta)<1-c/2\big)\ge1-\exp\big(-e^{(\log\log n)^2}\big).
$$
A weaker bound of this kind was obtained already in~\cite{QuenchedVoronoi}, and put to use in the work of Vanneuville~\cite{vanneuville1,vanneuville2,vanneuville3} in a further study of quenched and annealed properties of Poisson Voronoi percolation.


%
%

\subsection{Spectral techniques, a brief comment}

It is worth to mention that most studies of noise sensitivity in percolation has up to this point (in part) rested upon spectral techniques and discrete Fourier analysis. Spectral techniques is the basis for both the BKS Theorem and the original Schramm-Steif Revealment Theorem (which we haven't stated in full here), and consequently for any study of noise sensitivity that rests on either of these two results. While spectral techniques have proven to be very powerful, their strength sometimes come at the cost of intuition. The first proof of noise sensitivity for Bernoulli percolation on $\Z^2$ without the use of spectral techniques came only recently, in the work of Tassion and Vanneuville~\cite{TV}.

In the appendix we provide a proof of the version of the Revealment Theorem stated in~\eqref{eq:SS} which is probabilistic in nature. We emphasise that this version of the theorem is weaker than the original, and on its own not sufficient to prove noise sensitivity e.g.\ for Bernoulli percolation on $\Z^2$. However, this version of the Revealment Theorem was sufficient, together with the variance-influence bound in~\eqref{eq:efronstein}, to settle the noise-related conjecture regarding Voronoi percolation from~\cite{BKS99}. Together with an exponential inequality substituting~\eqref{eq:efronstein}, this version of the Revealment Theorem will be sufficient also here. Hence, there proofs or the results of this paper, as well as the main results of~\cite{QuenchedVoronoi}, are probabilistic in nature and does not rely on spectral techniques.

\subsection{Open problems}

The conjecture from~\cite{BKS99}, which has inspiered this work, concerns Poisson Voronoi percolation in $\R^2$. There are other natural and well-studied models for percolation in the continuum for which similar conjectures could be posed, e.g.\ Poisson Boolean and confetti percolation. Let us exemplify what such a conjecture would entail in the context of Boolean percolation.

Consider a Poisson point process of intensity $\lambda>0$ in $\R^2$. Centered at the points we place discs with radii drawn independently from some probability distribution. (The radii could, for instance, take the values 1 and 2 with equal probability.) It is well-known that there exists a critical value $\lambda_c\in(0,\infty)$ at which the probability of a horizontal crossing of a large square consisting of overlapping discs remains bounded away from zero and one; see~\cite{ahltastei}. In analogy to the results reported on here, we conjecture that given the positions of the points, but not the radii of the discs, we have asymptotically no information of whether a large square is crossed or not.

While most of the techniques of this paper and~\cite{QuenchedVoronoi} should apply also to Poisson Boolean percolation, settling the conjecture remains an open problem in this setting. The main reason for this is that we are currently unaware of how to carry out the key first step that corresponds to proving a version\footnote{For a precise statement of the version of~\eqref{eq:nondeg} proved in~\cite{QuenchedVoronoi}, see~\eqref{eq:q_rsw}-\eqref{eq:q_rsw_halfplane} below.} of~\eqref{eq:nondeg}. The proof (see~\cite[Section~3]{QuenchedVoronoi}) in the context of Voronoi percolation crucially rests on a colour-switching argument that does not translate easily to Boolean or confetti percolation.

\subsection{Outline of the paper}

The remainder of this paper is organized as follows. In Section~\ref{Comparison}, we establish some lemmas that will allow us to pass between the two models. In Section~\ref{s:crossing}, we discuss the box-crossing property for Poisson Voronoi percolation on which our work builds. In Section~\ref{ArmRevealment}, we provide the first of the two main pillars that our argument builds, and provide a bound on the sum of influences squared via the analysis of an algorithm with low revealment. In Section~\ref{EfronStein}, we provide the second pillar, deducing a bound on the deviations of $\P(H_R|\eta)$ from its mean via the bound on the sum of influences squared and an exponential inequality improving upon~\eqref{eq:efronstein}. 

In Section~\ref{ImprovedResult}, we first prove a preliminary version of Theorem~\ref{thm:optimized} for rectangles $R'\subseteq R$ that are contained strictly within $R$, and in Section~\ref{s:boundary}, we prove a preliminary version of Theorem~\ref{thm:optimized} for arbitrary $R'\subseteq R$. It is not essential to first consider rectangles in the bulk of $R$ in Section~\ref{ImprovedResult}, and then arbitrary rectangles in Section~\ref{s:boundary}, but we believe that it will make the presentation easier to follow. In Section~\ref{s:optimized}, we complete the proof of Theorem~\ref{thm:optimized} as a corollary of the result obtained in Section~\ref{s:boundary}. Finally, in Section~\ref{sec:poptimized} we prove Theorem~\ref{thm:poptimized}, the analogous result in the Poisson version of the model.

\section{Comparison between the two models}\label{Comparison}

In this section we will prove a few lemmas comparing the binomial distribution to the Poisson distribution. Although the lemmas merely compare the two distributions, they will be used to compare the two models $\P_R$ and $\Po$, and are therefore stated in these terms.  We shall generally make this comparison with $R$ being a rectangle of area $n$, because in this case, $n$, the number of points in the $\P_R$ model is equal to the expected number of points in the Poisson model $\Po$.

Let $R'\subseteq R$ be two rectangles. We will let $\mathcal{F}_{R'}$ denote the sub-$\sigma$-algebra of events that are measurable with respect to the restriction of $\eta$ to $R'$. Let $A_m$ denote the event that the restriction of $\eta$ to $R'$ has size $m$, i.e.\ $A_m:=\{|\eta\cap R'|=m\}$. Note, in particular, that for any $m=0,1,\ldots,n$ the conditional measures $\P_R(\,\cdot\,|A_m)$ and $\Po(\,\cdot\,|A_m)$ coincide on $\mathcal{F}_{R'}$. This will be used repeatedly below, and we will simply write $\P(\,\cdot\,|A_m)$ for either of the two.

Our first lemma gives an upper bound on the probability in the binomial model as a multiple of that of the Poisson model, and states that the two probabilities differ by $o(1)$ as $n\to \infty$.

\begin{lemma}\label{l:p_to_b}
Let $R$ be any rectangle of area $n$ and let $R'\subseteq R$ be any rectangle of area $n/2$. For any event $E\in\mathcal{F}_{R'}$ we have
$$
\P_R(E)\le 2\,\Po(E).
$$
\end{lemma}

\begin{proof}
We will show, for all $m=0,1,...,n$, that
\begin{equation}\label{eq:pibound}
    \P_R(A_m)\le 2\,\Po(A_m).
\end{equation}
From~\eqref{eq:pibound} it follows that
$$
\P_R(E)=\sum_{m=0}^n\P(E|A_m)\P_R(A_m)\le2\sum_{m=0}^n\P(E|A_m)\Po(A_m)\le2\,\Po(E).
$$
We thus prove~\eqref{eq:pibound}.

Note that since $R'$ has area $n/2$, which is half of that of $R$, we have
\begin{equation}\label{eq:pidef}
\pi_m:=\frac{\P_R(A_m)}{\Po(A_m)}={n\choose m}\frac{2^{-n}m!}{(n/2)^{m}e^{-n/2}}=\frac{n!}{(n-m)!}\frac{e^{n/2}}{2^{n-m}n^m}.
\end{equation}
Note further that for $m\ge n/2$ we have
$$
\frac{\pi_{m+1}}{\pi_m}=2\frac{n-m}{n}\le1,
$$
so $\pi_m$ takes it maximum for $m\le n/2$. Moreover, for $m\le n/2$, Stirling approximation gives
$$
\pi_m\le\frac{e}{\sqrt{2\pi}}\Big(\frac{n}{n-m}\Big)^{n-m+1/2}\frac{e^{n/2-m}}{2^{n-m}}\le\frac{e}{\sqrt{2\pi}}\sqrt{\frac{n}{n-m}}\Big(1-\frac{n/2-m}{n-m}\Big)^{n-m}e^{n/2-m}\le\frac{e}{\sqrt{\pi}}\le2,
$$
where we in the second-to-last step have used that $(1-x)\le e^{-x}$.
This proves~\eqref{eq:pibound}.
\end{proof}


The next lemma provides a bound in the opposite direction.

\begin{lemma}\label{l:p_to_b_lower}
Let $R$ be any rectangle of area $n$ and let $R'\subset R$ be any rectangle of area $n/2$. For every $E\in\mathcal{F}_{R'}$ and $n\ge 36/\Po(E)$ we have
$$
\P_R(E)\ge \frac{\Po(E)}{4}e^{-3/\Po(E)}.
$$
\end{lemma}

\begin{proof}
Let $0<a\le\sqrt{n}/6$ be a constant and let $I:=\{m\in\N:|m-n/2|\le a\sqrt{n}\}$. Then Chebyshev's inequality gives
\begin{equation}\label{eq:abound1}
\Po\bigg(\bigcup_{m\not\in I}A_m\bigg)=\Po\Big(\big||\eta\cap R'|-n/2\big|>a\sqrt{n}\Big)\le\frac{1}{2a^2}.
\end{equation}
We will show that for $m\in I$ and $a\le\sqrt{n}/6$ we have
\begin{equation}\label{eq:abound2}
    \P_R(A_m)\ge\frac12e^{-3a^2}\Po(A_m).
\end{equation}
From~\eqref{eq:abound1} and~\eqref{eq:abound2} it follows that
$$
\P_R(E)=\sum_{m=1}^n\P(E|A_m)\P_R(A_m)\ge\frac12e^{-3a^2}\sum_{m\in I}\P(E|A_m)\Po(A_m)\ge\frac12e^{-3a^2}\Big(\Po(E)-\frac{1}{2a^2}\Big).
$$
Setting $a=\Po(E)^{-1/2}$ then gives the claimed bound for $n\ge36/\Po(E)$.

Recall the definition of $\pi_m$ in~\eqref{eq:pidef}. Stirling approximation similarly gives the lower bound
$$
\pi_m\ge\frac{\sqrt{2\pi}}{e}\sqrt{\frac{n}{n-m}}\Big(1-\frac{n/2-m}{n-m}\Big)^{n-m}e^{n/2-m}.
$$
Using that $1-x\ge\exp(-x^2-x)$ for $x\le1/2$, we obtain for $m\in I$ and $\sqrt{n}\ge6a$ that
$$
\pi_m\ge\frac{\sqrt{2\pi}}{e}\exp\Big(-\frac{(n/2-m)^2}{n-m}\Big)\ge\frac{\sqrt{2\pi}}{e}\exp\Big(-\frac{a^2n}{n/2-a\sqrt{n}}\Big)
\ge\frac{\sqrt{2\pi}}{e}\exp(-3a^2),
$$
which proves~\eqref{eq:abound2}.
\end{proof}

Our final lemma gives a bound on the Poisson model in terms of the binomial model.

\begin{lemma}\label{l:b_to_p}
For every rectangle $R$, $E\in\mathcal{F}_{R}$ and $N\ge1$ we have
$$
\Po(E)\le\sup_{n\ge N}\P_R(E)+\Po(|\eta\cap R|<N).
$$
\end{lemma}

\begin{proof}
This time, we let $B_n:=\{|\eta\cap R|=n\}$. Then,
$$
    \Po(E)=\sum_{n\ge0}\P(E|B_n)\Po(B_n)=\sum_{n\ge0}\P_R(E)\Po(B_n)\le\sup_{n\ge N}\P_R(E)+\Po\bigg(\bigcup_{n<N}B_n\bigg),
$$
as required.
\end{proof}

\section{Crossing probabilities}\label{s:crossing}

RSW techniques, which make it possible to extend a crossing of a square to a crossing of a rectangle, are central in order to understand critical behaviour of planar percolation models.
For Voronoi percolation, an RSW theory was developed by Tassion~\cite{TassionRSW}, building on preliminary work of Bollob\'as and Riordan~\cite{bolrio06}. As a consequence, Tassion~\cite{TassionRSW} derived the box-crossing property for Poisson Voronoi percolation on $\R^2$, saying that for every $\theta>0$ there exists $c=c(\theta)>0$ such that for all $\rho\in[\theta,1/\theta]$ and $n\ge1$
\begin{equation}\label{eq:rsw}
c<\Po\big(H_{R(\rho,n)}\big)<1-c.
\end{equation}
The box-crossing property was extended in~\cite{QuenchedVoronoi} to tessellations of the half-plane $\mathbb{H}:=[0,\infty)\times\R$, in which boundary effects arise. Let
$$
R_0(\rho,n):=\big[0,\sqrt{\rho n}\big]\times\Big[-\frac12\sqrt{n/\rho},\frac12\sqrt{n/\rho}\Big]
$$
denote the rectangle $R(\rho,n)$ shifted so that its left side coincides with the vertical axis, and for every $R\subseteq\mathbb{H}$ let $H^\ast_R$ denote the event of a red horizontal crossing of $R$ with respect to the Voronoi tessellation generated by the restriction of $\eta$ to $\mathbb{H}$. In~\cite{QuenchedVoronoi}, it was proven that for every $\theta>0$ there exists $c'=c'(\theta)>0$ such that for all $\rho\in[\theta,1/\theta]$ and $n\ge1$
\begin{equation}\label{eq:rsw_halfplane}
c'<\Po\big(H^\ast_{R_0(\rho,n)}\big)<1-c'.
\end{equation}

As mentioned in the introduction, as a key step in settling the conjecture from~\cite{BKS99}, the authors of~\cite{QuenchedVoronoi} derived a preliminary `quenched' box-crossing estimate: For every $\theta>0$ and $\eps>0$ there exists $\delta>0$ such that for all $\rho\in[\theta,1/\theta]$ and $n\ge1$
\begin{equation}\label{eq:q_rsw}
\Po\Big(\P\big(H_{R(\rho,n)}|\eta\big)\in(\delta,1-\delta)\Big)>1-\eps.
\end{equation}
The analogous statement was shown to hold also for the half-plane, in that for every $\theta>0$ and $\eps>0$ there exists $\delta>0$ such that for all $\rho\in[\theta,1/\theta]$ and $n\ge1$
\begin{equation}\label{eq:q_rsw_halfplane}
\Po\Big(\P\big(H^\ast_{R_0(\rho,n)}|\eta\big)\in(\delta,1-\delta)\Big)>1-\eps.
\end{equation}
The estimates in~\eqref{eq:q_rsw} and~\eqref{eq:q_rsw_halfplane} will be key also in the present paper, and will be used to get the induction argument, in the proof of our main theorem, started.

The box-crossing property for the half-plane (that is~\eqref{eq:rsw_halfplane}) effectively provides information about tessellations of other subsets of $\R^2$. We illustrate this fact, and the use of the comparison lemmas from the previous section, by establishing a box-crossing property for the binomial model.

\begin{proposition}\label{p:rsw}
For every $\theta>0$ there exists $c''=c''(\theta)$ and $K=K(\theta)$ such that for all $\rho\in[\theta,1/\theta]$, $n\ge K$ and (axis parallel) rectangle $R'\subseteq R(\rho)$ of area at least $\theta$ we have
$$
c''<\P_{R(\rho)}\big(H_{R'}\big)<1-c''.
$$
\end{proposition}

\begin{proof}
We will consider tessellations of $R(\rho,n)$ is order to allow for comparisons between $\P_{R(\rho,n)}$ and $\Po$.
Given a rectangle $R\subseteq R(\rho,n)$, let $H^{\ast\ast}_R$ denote the event of a red horizontal crossing of $R$, and $V^{\ast\ast}_R$ the event of a red vertical crossing of $R$, in the Voronoi tessellation of the restriction of $\eta$ to $R(\rho,n)$.

In the tessellation of a rectangle, we will see boundary effects from its four sides. The restriction of this tessellation to a rectangle $R'\subseteq R(\rho,n)$ that is aligned with e.g.\ the left side of $R(\rho,n)$, but do not touch the other three sides of $R(\rho,n)$, will (for large $n$) coincide with that of a half-plane. This will allow us to deduce the claim of the proposition from~\eqref{eq:rsw_halfplane}.

We shall first consider $R'$ of the form $R'=[-\frac12\sqrt{\rho n},\frac12\sqrt{\rho n}]\times[-\frac a2\sqrt{n/\rho},\frac a2\sqrt{n/\rho}]$, where $a\in(0,1/4),$ and $a<\rho$ is a fixed constant. Denote by $S'$ the square $[-\frac a2\sqrt{n/\rho},\frac a2\sqrt{n/\rho}]^2$, and denote by $R'_1$ the rectangle $[-\frac12\sqrt{\rho n},\frac a2\sqrt{n/\rho}]\times[-\frac a2\sqrt{n/\rho},\frac a2\sqrt{n/\rho}]$ and $R'_2$ the reflection of $R'_1$ in the vertical axis. Note that a red horizontal crossing of each of $R'_1$ and $R'_2$, together with a red vertical crossing of $S'$, results in a red horizontal crossing of $R'$. That is,
$$
H^{\ast\ast}_{R'_1}\cap H^{\ast\ast}_{R'_2}\cap V^{\ast\ast}_{S'}\subseteq H^{\ast\ast}_{R'},
$$
so that by the FKG-inequality for Poisson Voronoi percolation (see e.g.~\cite[Lemma~8.3]{bollobás_riordan_2006})
\begin{equation}\label{eq:R'lower}
\Po\big(H^{\ast\ast}_{R'}\big)\ge\Po\big(H^{\ast\ast}_{R'_1}\big)\,\Po\big(H^{\ast\ast}_{R'_2}\big)\,\Po\big(V^{\ast\ast}_{S'}\big)=\Po\big(H^{\ast\ast}_{R'_1}\big)^2\,\Po\big(V^{\ast\ast}_{S'}\big).
\end{equation}

Next, cover\footnote{By a \emph{covering} of a set $S\subseteq\R^2$ we refer to a collection of subsets of $S$ whose union equals $S$.} $S'$ by at most $2a^2\sqrt{n}/\rho$ squares of area $\sqrt{n}$, and let $E_n$ denote the event that each square in the covering contains a point of $\eta$. Note that on $E_n$, the events $V_{S'}$ and $V^{\ast\ast}_{S'}$ are determined by the restriction of $\eta$ to a $2n^{1/4}$-neighbourhood of $S'$, and thus coincide. It follows that
$$
\Po\big(V^{\ast\ast}_{S'}\big)\ge\Po\big(V^{\ast\ast}_{S'}\cap E_n\big)=\Po\big(V_{S'}\cap E_n\big)\ge\Po\big(V_{S'}\big)-\Po(E_n)\ge\frac12-\frac{2a^{2}}{\rho}\sqrt{n}e^{-\sqrt{n}}\ge\frac14,
$$
for large $n$ (depending on $a$). An analogous argument, together with~\eqref{eq:rsw_halfplane}, shows that there exists $c'=c'(a,\theta)>0$ such that $\Po(H^{\ast\ast}_{R'_1})\ge c'/2$ for large $n$ (depending on $a$ and $\theta$). By~\eqref{eq:R'lower} we obtain that $\Po(H^{\ast\ast}_{R'})\ge (c'/2)^2/4$ for all large $n$.

Finally, we cover $R'$ by at most $2a\sqrt{ n}$ squares of area $\sqrt{n}$, and let $F_n$ denote the event that each square in the cover contains a point of $\eta$. Again, $\Po(H^{\ast\ast}_{R^\prime}\cap F_n)\ge (c')^2/32$ for large $n$, and Lemma~\ref{l:p_to_b_lower} gives
$$
\P_{R(\rho,n)}\big(H_{R'}\big)=\P_{R(\rho,n)}\big(H^{\ast\ast}_{R'}\big)\ge\P_{R(\rho,n)}\big(H^{\ast\ast}_{R'}\cap F_n\big)\ge c'',
$$
where $c'':=\frac{(c')^2}{128}e^{-96/(c')^2}$.

We have now established the lower bound of the proposition for $R'$ as defined above. The same proof gives the same lower bound also for vertical translates $R''=(0,\frac{b}{2}\sqrt{n/\rho})+R'$ of $R'$, as long as $b\in[-(1-2a),1-2a]$, so that $R''$ do not touch the top or bottom of $R(\rho,n)$. Since for any rectangle $R\subseteq R(\rho,n)$ of area at least $\theta$ will may choose $a$ and $b$ (with $a>0$ uniformly) such that a horizontal crossing of $R''$ implies a horizontal crossing of $R$, this establishes the lower bound of the proposition.

The upper bound follows from an analogous argument, by considering vertical crossings of $R(\rho,n)$.
\end{proof}

\section{Bound on the sum of influences squared}\label{ArmRevealment}

In this section we show how to obtain a bound on the sum of influences squared from a bound on the crossing probability of a rectangle. We will follow a standard approach, and our exposition is close to that of~\cite{QuenchedVoronoi}, with the exception that we quantify with greater care the error bounds obtained.

\begin{proposition}\label{pt1}
Suppose that there exist constants $c,\alpha>0$ and $k, L \geq 1$ such that for all $n \geq L$
\begin{equation}\label{eqn:stp1}
 \Po\Big(c<\P(H_{R(3,n)}|\eta)<1-c\Big) \geq 1 - e^{-\alpha(\log n)^{k}}.
\end{equation}
Then, for every $\theta\in(0,1)$, there exist constants $\eps=\eps(c)$, $L'=L'(\alpha,k)$ and $L''=L''(\theta)$ such that for all $\rho \in [\theta,1/\theta]$ and $n\geq\max\{L^2,L',L''\}$
$$\P_{R(\rho)}\bigg( \sum_{j=1}^{n} \Inf_j\big(H_{R(\rho,1/4)}|\eta\big)^2\geq n^{-\eps}\bigg)\leq e^{-\frac{\alpha}{200}(\log n)^{k+1}}.$$
\end{proposition}

\begin{remark}\label{r:sq}
In applications of Proposition~\ref{pt1} the maximum will generally be obtained by the $L^2$, so that we may take $n=L^2$.  See Remark~\ref{r:N} for details.
\end{remark}

To prove the proposition, we will rely on the Revealment Theorem from~\cite{SS10} to bound the sum of influences squared in terms of the revealment of a random algorithm. To bound the revealment we need the following bound on the so-called one-arm event. Let $B(u,d)$ denote the $\ell_\infty$-ball $u+[-d,d]^2$, centered at $u$ with radius $d$. Let $V_u(a,b)$ denote the event that there is a red path (contained in $R(\rho,n)$) connecting $B(u,a)$ to $B(u,b)^c$.

\begin{lemma}\label{OneArmImp}
Suppose that there exist constants $c,\alpha>0$ and $k, L \geq 1$ such that~\eqref{eqn:stp1} holds for all $n \geq L$. Then, there exist constants $\eps=\eps(c)$, $L'(\alpha,k)$ and $L''(\theta)$ such that the following holds: For every $\rho\in[\theta,1/\theta]$ and $u\in R(\rho,n/4)$ we have for $n\ge \max\{L^2,L'(\alpha,k),L''(\theta)\}$ that
$$
\P_{R(\rho,n)}\Big(\P\big(V_u(n^{1/4}, n^{1/3})\big|\eta\big)\geq n^{-\eps}\Big)\leq e^{-\frac{\alpha}{100} (\log n)^{k+1}}.
$$
\end{lemma}

As mentioned in the introduction, it will sometimes be convenient to work with the Poisson model, due to its spatial independence property. For this reason we shall work with the Poisson model for the bulk of the proof of the lemma.
However, to better fit our later needs, we convert to the binomial model in the end of the proof.
To facilitate the comparison between the two models, the lemma is stated in terms of a rectangle of area $n$.

\begin{proof}[Proof of Lemma~\ref{OneArmImp}]
 Fix $u\in R(\rho,n/4)$. For each $j\in \N$, let $A_{j}$ denote the square annulus, centered at $u$, with inner side-length $7^{j}$ and outer side-length $3\cdot7^{j}$, so that $A_j=u+[-\frac327^j,\frac327^j]\setminus[-\frac127^j,\frac127^j]$. Let $O_{j}$ be the event that there is no red path connecting the inner and outer boundaries of the annulus $A_j$, which by duality means that there exists a blue circuit in $A_{j}$ surrounding $u$. Let 
 $$ J := \Big\{ j\in \N: n^{1/4} \leq 7^j \leq  \frac12n^{1/3}\Big\}. $$
 For $j\in J$, let $D_{j}'$ denote the event that $\P(O_{j}|\eta)>c^4$, where $c>0$ is the constant in~\eqref{eqn:stp1}, and let $D_{j}''$ denote the event that for every $z\in A_{j}$ there exists some point in $\eta$ at $\ell_2$-distance at most $n^{1/6}$ from $z$. Let
 $$D_{j}:= D_{j}' \cap D_{j}'',$$
 and note that the events $D_{j}$ are independent with respect to the Poisson model; $D_j$ depends on the restriction of $\eta$ to a $n^{1/6}$-neighbourhood of $A_j$, and for different $j$ these regions are disjoint.
 
 For later reference, we specify $L'(\alpha,k)$ to be the least integer such that for all $n\ge L'(\alpha,k)$ we have
 \begin{equation}\label{eq:N(a,k)}
 18n^{1/3}\le e^{\frac14n^{1/3}}\quad\text{and}\quad 3000\le \alpha(\log n)^k\le\frac14n^{1/3}.
 \end{equation}
 Next, observe that $A_{j}$ can be covered by $4$ rectangles with aspect ratio $3:1$, in such a way that (blue) crossings of each of these rectangles imply a (blue) circuit in $A_j$. Thus, by the (Harris-)FKG inequality, combined with condition~$\eqref{eqn:stp1}$ and the union bound, we have 
 $$ \Po\big( \P(O_{j}|\eta)\leq c^4\big) \leq 4e^{-\alpha(\log n)^{k}}$$
 for all $n\geq L^2$. (Note that for $n\ge L^2$ the area $3\cdot 7^{2j}\ge3n^{1/2}$ of each of these rectangles exceeds $L$).
 We further observe that we may cover\footnote{Recall that by a covering of a set $A\subseteq\R^2$ we refer to a collection of subsets of $A$ whose union equals $A$.} $A_j$ by at most $18n^{1/3}$ squares of area $\frac12n^{1/3}$, and that the event $D_j''$ may fail only if one of these squares contains no point of $\eta$. Consequently, using the union bound, we have
 \begin{equation}\label{ineq0}
  \Po(D_{j}'')\geq1-18n^{1/3}e^{-\frac12n^{1/3}}
  \geq 1-e^{-\frac14n^{1/3}}
  \geq 1-e^{-\alpha(\log n)^{k}},
  \end{equation}
 for all $n\ge L'(\alpha,k)$. So, $\Po(D_{j})\geq 1-5e^{-\alpha(\log n)^{k}}$, for $n\geq \max\{L^2,L'(\alpha,k)\}$.

 Let $J^\ast=J^\ast(\eta)$ denote the subset of indices $j\in J$ for which the event $D_j$ occurs, and let $D^{*}$ be the event that $|J^\ast|\ge|J|/2$. Since $|J|\ge\frac{1}{24}\log n$, we have on the event $D^{*}$ that
  $$\P\Big(V_u(n^{1/4}, n^{1/3})|\eta\Big)\leq \P\bigg(\bigcap_{j\in J}O^{c}_{j}\,\bigg|\,\eta\bigg)\le\P\bigg(\bigcap_{j\in J^\ast}O^{c}_{j}\,\bigg|\,\eta\bigg)=\prod_{j\in J^\ast}\P(O^{c}_{j}|\eta)\leq (1-c^4)^{|J|/2}\leq n^{-\eps},$$
  for some $\eps=\eps(c)>0$.
Moreover, since the events $D_j$ are independent with respect to $\Po$, we have
 \begin{equation}\label{ineq1}
 \Po(D^{*})\geq 1- 2^{|J|}\big(5e^{-\alpha(\log n)^{k}}\big)^{|J|/2}\geq 1-20^{\frac{1}{48}\log n}e^{-\frac{\alpha}{48}(\log n)^{k+1}}\geq 1 -e^{-\frac{\alpha}{96}(\log n)^{k+1}},
 \end{equation}
 for all $n\geq \max\{L^2,L'(\alpha,k)\}$ (so that, in particular, $\alpha(\log n)^k\ge6$).
 
 Finally, observe that, for some $L''(\theta)$, we have $D^{*}\in \F_{R(\rho,n/2)}$ for $n\ge L''(\theta)$ (taking $n\ge(16/\theta)^2$ will suffice uniformly over $\rho\in[\theta,1/\theta]$). By Lemma~$\ref{l:p_to_b}$ we thus obtain (by the choice of $L'(\alpha,k)$) that
 $$\P_{R(\rho,n)}\big(D^{*}\big)\geq1- 2 e^{-\frac{\alpha}{96} (\log n)^{k+1}}\geq1-e^{-\frac{\alpha}{100} (\log n)^{k+1}},$$
  for all $n\ge\max\{L^2,L'(\alpha,k),L''(\theta)\}$.
 \end{proof}
 
\begin{proof}[Proof of Proposition~\ref{pt1}]
We will rely on the Schramm-Steif Revealment Theorem, as stated in~\eqref{eq:SS}. To this end, let $\eta\sim\P_{R(\rho)}$ be a set of $n$ points.
Suppose that $\mathcal{A}$ is a random algorithm that, given $\eta$, queries the cells of the Voronoi tiling for their colour. We shall denote by $\delta(\mathcal{A}|\eta)$ the (conditional) revealment of $\mathcal{A}$ with respect to $\P(\,\cdot\,|\eta)$. If $\mathcal{A}$ determines whether the event $H_{R(\rho,1/4)}$ has occurred or not (with probability one), then~\eqref{eq:SS} gives that, almost surely,
\begin{equation}\label{eq:SSeta}
    \sum_{j=1}^n\Inf_j\big(H_{R(\rho,1/4)}|\eta\big)^2\le\delta(\mathcal{A}|\eta).
\end{equation}

The algorithm that we describe next is a version, due to Gady Kozma, of the algorithm devised in~\cite{BKS99}, which we have further adapted to the continuum.

\begin{definition}[The BKS-Kozma algorithm]
Let $\mathcal{A}$ be the algorithm that, given $\eta$, queries the cells of the Voronoi tessellation of $R(\rho)$ for their colour, in order to detect a crossing of $R(\rho,1/4)$, as follows:
\begin{enumerate}[\quad (i)]
    \item Choose a point in the middle third of the top side of $R(\rho,1/4)$ uniformly at random, and let $\ell$ denote the vertical line segment that connects that point to the bottom side of $R(\rho,1/4)$. Query every cell that intersects $\ell$, and declare these cells explored.
    \item Query any unexplored cell that has non-empty intersection with $R(\rho,1/4)$, and that is adjacent to a previously explored red cell. Repeat this step until no further cells of this kind exist.
\end{enumerate}
\end{definition}

Note that the algorithm detects all cells that are connected (within $R(\rho,1/4)$) to the line segment $\ell$ by a red path. In particular, the algorithm determines whether $H_{R(\rho,1/4)}$ occurs or not.

The revealment of the above algorithm relates to the one-arm event in the following sense. Cover $R(\rho,1/4)$ by $m\le8\sqrt{n}$ squares of area $1/(8\sqrt{n})$, and let $u_1,u_2,\ldots,u_m$ denote the centers of these squares.

\begin{claim}\label{claim:revealment}
There exists $L'''=L'''(\theta)$ so that for $n\ge L'''$, with probability at least $1-\exp(-\sqrt{n}/16)$, we have
\begin{equation}\label{eq:revealment}
\delta(\mathcal{A}|\eta) \leq \max_{1\le j\le m}\P\big(V_{u_j}(n^{-1/4},n^{-1/6})\big|\eta\big) + O(n^{-1/6}).
\end{equation}
\end{claim}

\begin{proof}[Proof of claim]
Let $E_n$ be the event that every cell in the Voronoi tessellation of $R(\rho)$ has radius at most $\frac12n^{-1/4}$.  We shall prove that~\eqref{eq:revealment} holds on the event $E_n$, and then bound the probability that $E_n$ fails. 

On the event $E_n$, we argue that the probability for the cell of a point $u\in\eta$ to be queried is at most $P\big(V_{u_j}(n^{-1/4},n^{-1/6})\big|\eta\big) + O(n^{-1/6})$, where $u_j$ is the center of the square which contains $u$ (according to the above covering of $R(\rho,1/4)$). Indeed, either $u$ is within distance $2n^{-1/6}$ of $\ell$, which occurs with (conditional) probability at most $O(n^{-1/6})$, or there is a red path from a neighbouring cell of $u$ to $\ell$. Since the cell of $u$ has radius at most $\frac12n^{-1/4}$, this implies that the event $V_{u_j}(n^{-1/4},n^{-1/6})$ occurs, which has probability  $\P(V_{u_j}(n^{-1/4},n^{-1/6})|\eta)$, as required.

It remains to bound the probability that $E_n$ fails.
So, extend the covering of $R(\rho,1/4)$ to a covering of squares of all of $R(\rho)$ consisting of at most $16\sqrt{n}$ squares of area $1/(8\sqrt{n})$. Note that $E_n$ occurs if each square in the covering of $R(\rho)$ contains a point of $\eta$. A given square is empty with $\P_{R(\rho)}$-probability
$$
\big(1-1/(8\sqrt{n})\big)^n\le \exp(-\sqrt{n}/8).
$$
Hence, the conclusion of the lemma follows by the union bound, for large $n$.
\end{proof}

Using Claim~\ref{claim:revealment} and Lemma~\ref{OneArmImp} we obtain the following bound on the revealment of $\mathcal{A}$: There exists $\eps=\eps(c)<1/6$ such that
\begin{align*}
    \P_{R(\rho)}\big(\delta(\mathcal{A}|\eta)>2n^{-\eps}\big)&\le\P_{R(\rho)}\Big(\max_{1\le j\le m}\P\big(V_{u_j}(n^{-1/4},n^{-1/6})\big|\eta\big)>n^{-\eps}\Big)+e^{-\sqrt{n}/16}\\
    &\le\sum_{j=1}^m\P_{R(\rho)}\Big(\P\big(V_{u_j}(n^{-1/4},n^{-1/6})\big|\eta\big)>n^{-\eps}\Big)+e^{-\sqrt{n}/16}\\
    &\le(8\sqrt{n}+1)\exp\Big(-\frac{\alpha}{100}(\log n)^{k+1}\Big),
\end{align*}
for $n\ge\max\{L^2,L'(\alpha,k),L''(\theta),L'''(\theta)\}$. Using~\eqref{eq:N(a,k)}, the conclusion now follows from~\eqref{eq:SSeta}.
\end{proof}

\begin{remark}\label{r:a}
Proposition~\ref{pt1} is stated for $k\ge1$, but holds also for $k=0$ provided that $\alpha\ge6$, in order for~\eqref{ineq1} to hold.
\end{remark}

\begin{remark}\label{r:N}
Recall that $L'(\alpha,k)$ was defined as the least integer such that~\eqref{eq:N(a,k)} holds for all $n\ge L'(\alpha,k)$.
We shall later apply Proposition~\ref{pt1} iteratively, with $\alpha$ and $L$ of the form $\alpha=\gamma^{k-1}$ and $L=N^{2^{k-1}}$, where $\gamma\in(0,1)$ and $N\ge1$ are constants. For the iterative scheme to be useful, we need to control the rate at which $L'(\gamma^{k-1},k)$ grows with $k$. We claim that for every $\gamma\in(0,1)$ there exists $N\ge1$ such that
\begin{equation}\label{eq:Lbound}
L'(\gamma^{k-1},k)\le N^{2^{k-1}}\quad\text{for all }k\ge1.
\end{equation}

In order to verify~\eqref{eq:Lbound}, note that for the lower inequality in~\eqref{eq:N(a,k)} it will suffice that $\gamma\log N\ge3000$. We claim that if $N$ is the least integer such that $\log N\ge 96$, then also $4(\log n)^k\le n^{1/3}$ holds for all $k\ge1$ and $n\ge N^{2^{k-1}}$. Indeed, for $n=N^{2^{k-1}}$ the left-hand side is bounded by $2^{k(k-1)+5k+2}$, whereas the right-hand side is at least
$$
\exp\Big(\frac132^{k-1}\log N\Big)\ge\exp(2^{k+4})\ge\frac{(2^{k+4})^{2k}}{(2k)!}\ge2^{2k^2+4k}\frac{2^{k+4}}{(2k)^{2k}}\ge2^{k(k-1)+5k+2}.
$$
Moreover, increasing $n$ by a factor $c$ increases the right-hand side by a factor $c^{1/3}$, whereas the left-hand side increases by a factor
$$
\Big(1+\frac{\log c}{\log n}\Big)^k\le \exp\Big(k\,\frac{\log c}{\log n}\Big)\le \exp\Big(\frac{\log c}{\log N}\Big)\le \exp\Big(\frac{1}{32}\log c\Big)=c^{1/32}.
$$
This shows that, for any $\gamma\in(0,1)$, there exists $N=N(\gamma)$ such that $L'(\gamma^{k-1},k)\le N^{2^{k-1}}$.
\end{remark}

\section{Improved bound on deviations from the mean}\label{EfronStein}

In this section we will revisit the relation between variance and influence that in~\cite{QuenchedVoronoi} led to the Efron-Stein-like inequality in~\eqref{eq:efronstein}. We shall here aim for an exponential version of~\eqref{eq:efronstein} by estimating the moment generating function of the conditional probability of crossing a rectangle.

\begin{proposition}\label{pt2}
Let $\rho\in(0,\infty)$ be fixed and let $R'$ be any rectangle contained in $R(\rho)$. Note that $R(1/\rho)$ is the rectangle obtained from $R(\rho)$ by rotating the plane by $\pi/2$, and let $R''$ be the rectangle obtained from $R'$ by the same rotation. Suppose that there exist $\eps,\beta>0$ and $k,M\geq1$ such that for all $n\ge M$
\begin{align}\label{eqn:stp2}
 \P_{R(\rho)}\bigg(\sum_{j=1}^{n} \Inf_j(H_{R'}|\eta)^2\geq n^{-\eps} \bigg) &\leq e^{-\beta(\log n)^{k}},\\
  \P_{R(1/\rho)}\bigg(\sum_{j=1}^{n} \Inf_j(H_{R''}|\eta)^2\geq n^{-\eps} \bigg) &\leq e^{-\beta(\log n)^{k}}. \label{eqn:stp2dual}
 \end{align}
Then, there exists $M'=M'(\eps,\beta,k)$ such that for all $n\geq\max\{M,M'\}$ and $t>0$ we have
$$
\P_{R(\rho)}\Big( \big|\P(H_{R'}|\eta)-\P_{R(\rho)}(H_{R'})\big| \geq t \Big) \leq 4\exp\Big(-t\,\frac{\beta}{4}(\log n)^{k}\Big).
$$
\end{proposition}

In order to describe our approach, let $Z$ be any random variable whose moment generating function $\E[e^{\lambda Z}]$ exists (at least for small $\lambda>0$). Define for $\lambda>0$ the function
$$
F(\lambda):=\E\Big[e^{\lambda(Z-\E[Z])}\Big].
$$
Then, Markov's inequality yields the bound
$$
\P\big(Z-\E[Z]\ge t\big)\le F(\lambda)\cdot e^{-\lambda t},\quad\text{for }\lambda>0.
$$
This expression is the basis for various Chernoff-like concentration bounds. With $Z:=\P(H_{R'}|\eta)$, our aim will be to show that $F$ is bounded for suitable values of $\lambda>0$, based on the relation between variance and influence explored in~\cite{QuenchedVoronoi}.
That is, we will aim to specify $\lambda$ (as a function of $n$) as large as we can, while $F(\lambda)$ remains bounded by a constant.
This approach is in~\cite{HSE} accredited to Aida and Stroock~\cite{AS94}, and our presentation is influenced by the exposition in~\cite[pages 70-71]{HSE}.

\begin{lemma}
Let $Z:=\P(H_{R'}|\eta)$. For all $\lambda>0$ we have 
\begin{equation}\label{eqn:ineq1}
\V_{R(\rho)}\big(e^{\lambda Z/2}\big) \leq \frac{\lambda^{2}}{4}\E_{R(\rho)}\Big[e^{\lambda Z}\sum_{j=1}^{n}e^{\lambda\cdot\Inf_{j}(H_{R'}|\eta)}\Inf_{j}(H_{R'}|\eta)^2\Big].
\end{equation}
\end{lemma}

\begin{proof}
Label the points in $\eta$ arbitrarily by $\eta_1,\eta_2,\ldots,\eta_n$, and let $\mathcal{F}_m$ be the $\sigma$-algebra generated by the positions (which are independent and uniform within $R(\rho)$) of the first $m$ points. Consider the martingale $(q_m)_{m=1,2,\ldots,n}$ where $q_{m}:=\E\big[e^{\lambda Z/2}|\mathcal{F}_{m}\big]$.
By independence of martingale increments we have
\begin{equation}\label{eq:mart}
\V_{R(\rho)}\big(e^{\lambda Z/2}\big)= \sum_{m=1}^{n}\V_{R(\rho)}\big(q_{m}-q_{m-1}\big).
\end{equation}

Write $\eta^-$ for the configuration obtained from $\eta$ when $\eta_m$ is removed, and let $Z^-:=\P(H_{R'}|\eta^-)$. That is, $Z^-$ is the conditional probability that the colouring of the tessellation based on the $n-1$ points in $\eta^-$ results in a red crossing of $R'$. Following~\cite{QuenchedVoronoi}, we next claim that
\begin{equation}\label{eq:mart2}
\V_{R(\rho)}\big(q_{m}-q_{m-1}\big) \leq \E_{R(\rho)} \Big[ \big(e^{\lambda Z/2}-e^{\lambda Z^-/2}\big)^{2}\Big].
\end{equation}
To see this, note that by conditioning on $\mathcal{F}_{m-1}$, the conditional variance formula gives
$$
\V_{R(\rho)}\big(q_m-q_{m-1}\big)=\E_{R(\rho)}\big[\V(q_m|\mathcal{F}_{m-1})\big].
$$
Moreover,
$$
\V(q_m|\mathcal{F}_{m-1})=\V\Big(\E\big[e^{\lambda Z/2}-e^{\lambda Z^-/2}\big|\mathcal{F}_{m}\big]\Big|\mathcal{F}_{m-1}\Big)\le\E\Big[\E\big[e^{\lambda Z/2}-e^{\lambda Z^-/2}\big|\mathcal{F}_m\big]^2\Big|\mathcal{F}_{m-1}\Big],
$$
which by Jensen's inequality is at most $\E\big[(e^{\lambda Z/2}-e^{\lambda Z^-/2})^2\big|\mathcal{F}_{m-1}\big]$, and~\eqref{eq:mart2} follows.

We next claim that, almost surely, for all $\lambda>0$
\begin{equation}\label{eq:mart3}
\big|e^{\lambda Z/2}-e^{\lambda Z^-/2}\big| \leq \frac{\lambda}{2} \,e^{\lambda Z/2}\, e^{\lambda\cdot\Inf_{m}(H_{R'}|\eta)/2}\,\Inf_{m}(H_{R'}|\eta).
\end{equation}
To see this, first observe that, almost surely,
\begin{equation}\label{eq:exclude}
|Z-Z^-|=|\P(H_{R'}|\eta)-\P(H_{R'}|\eta^-)|\le\Inf_m(H_{R'}|\eta).
\end{equation}
This is due to the fact that adding a red point will only increase the probability for $H_{R'}$, and adding a blue point will only decrease the probability; the difference between the two is simply the influence of the newly added point.

We consider the case $Z\ge Z^-$ first, and observe that by the Mean Value Theorem
$$
e^{\lambda Z/2}-e^{\lambda Z^-/2} \,\leq\, \frac{\lambda}{2}\, e^{\lambda Z/2}\, |Z-Z^-|\,\leq\, \frac{\lambda}{2}\, e^{\lambda Z/2}\, \Inf_{m}(H_{R'}|\eta),
$$
almost surely. When $Z<Z^-$, we similarly obtain, almost surely, that
$$
e^{\lambda Z/2}-e^{\lambda Z^-/2} \leq\frac{\lambda}{2}\, e^{\lambda Z^-/2}\,|Z-Z^-|\leq \frac{\lambda}{2}\, e^{\lambda(Z+\Inf_m(H_{R'}|\eta))/2} \,\Inf_{m}(H_{R'}|\eta),
$$
which yields~\eqref{eq:mart3}.

The lemma now follows by combining~\eqref{eq:mart}-\eqref{eq:mart3}.
\end{proof}

\begin{proof}[Proof of Proposition~\ref{pt2}]
Let $Z:=\P(H_{R'}|\eta)$ and let $F(\lambda):=\E_{R(\rho)}\big[\exp\big(\lambda(Z-\E_{R(\rho)}[Z])\big)\big]$.
Note that by multiplying both sides in~\eqref{eqn:ineq1} by $\exp(-\lambda\E_{R(\rho)}[Z])$ we obtain the expression
\begin{equation}\label{eqn:ineq2}
F(\lambda)-F(\lambda/2)^2\,\le\,\frac{\lambda^2}{4}\,e^{-\lambda\E_{R(\rho)}[Z]}\,\E_{R(\rho)}\Big[e^{\lambda Z}\sum_{j=1}^{n}e^{\lambda\cdot\Inf_{j}(H_{R'}|\eta)}\,\Inf_{j}(H_{R'}|\eta)^2\Big].
\end{equation}
Let $B_n:=\big\{\sum_{j=1}^n\Inf_j(H_{R'}|\eta)^2>n^{-\eps}\big\}$. For $0<\lambda<n^{\eps/2}$, the right-hand side in~\eqref{eqn:ineq2} is further bounded by
$$
\frac{e \lambda^{2}}{4 n^{\eps}}F(\lambda)+\frac{\lambda^2}{4}\,e^{-\lambda\E_{R(\rho)}[Z]}\,\E_{R(\rho)}\Big[e^{\lambda Z}\sum_{j=1}^{n}e^{\lambda\cdot\Inf_{j}(H_{R'}|\eta)}\,\Inf_{j}(H_{R'}|\eta)^2\,\Ind_{B_n}\Big].
$$
Since influences are probabilities, they are bounded by $1$. For monotone events, the sum of influences squared are again bounded by $1$. (This follows e.g. by~\eqref{eq:SS}.) Consequently, we obtain the following further upper bound on the right-hand side of~\eqref{eqn:ineq2}
$$
\frac{e \lambda^{2}}{4 n^{\eps}}F(\lambda)+\frac{\lambda^2e^{2\lambda}}{4}e^{-\lambda\E_{R(\rho)}[Z]}\P_{R(\rho)}(B_n)\le\frac{\lambda^2}{4}\Big(\frac{e}{ n^{\eps}}+e^{2\lambda}\P_{R(\rho)}(B_n)\Big)F(\lambda).
$$
Under condition~\eqref{eqn:stp2}, the event $B_n$ occurs with probability at most $\exp(-\beta(\log n)^k)$. So, for $n\ge M$ and $0<\lambda<\frac{\beta}{3}(\log n)^k$ we note that, since $e^{-x}\le x^{-2}$ for $x>0$,
$$
e^{2\lambda}\P_{R(\rho)}(B_n)\le \exp\Big(\frac{2\beta}{3}(\log n)^k-\beta(\log n)^k\Big)\le\exp\Big(-\frac{\beta}{3}(\log n)^k\Big)\le\frac{9}{\beta^2(\log n)^{2k}}.
$$
Let $M'=M'(\eps,\beta,k)$ be the least integer such that $n^{\eps/2}\ge\beta(\log n)^k$ for all $n\ge M'(\eps,\beta,k)$.
With this, equation~\eqref{eqn:ineq2} yields that for $n\ge \max\{M,M'\}$ and $0<\lambda<\frac{\beta}{3}(\log n)^k$ we have
$$
F(\lambda)-F(\lambda/2)^2\le\frac{3\lambda^2}{\beta^2(\log n)^{2k}}F(\lambda),
$$
which after rearrangements takes the form
\begin{equation}\label{eqn:ineq3}
    F(\lambda)\le F(\lambda/2)^2\Big(1-\frac{3\lambda^2}{\beta^2(\log n)^{2k}}\Big)^{-1}.
\end{equation}

\begin{claim}
For $n\ge\max\{M,M'\}$ and $0<\lambda\le\frac{\beta}{4}(\log n)^k$ we have $F(\lambda)\le2$.
\end{claim}

\begin{proof}[Proof of claim]
Let $a=3(\beta(\log n)^k)^{-2}$. By iterated use of~\eqref{eqn:ineq3} we obtain
$$
F(\lambda)\,\le\, F(\lambda/2)^2(1-a\lambda^2)^{-1}\,\le\,\ldots\,\le\, F(\lambda/2^m)^{2^m}\prod_{j=0}^{m-1}\Big(1-\frac{a\lambda^2}{2^{2j}}\Big)^{-2^j}.
$$
Since $\log x\le x-1$, we find that
$$
\lim_{m\to\infty}2^m\log\big(F(\lambda/2^m)\big)\le\lim_{m\to\infty}2^m\big(F(\lambda/2^m)-1\big)=\lim_{x\downarrow 0}\lambda\, F'(x)=\lambda\,\E_{R(\rho)}\big[Z-\E_{R(\rho)}[Z]\big]=0.
$$
Consequently, and further using that $a\lambda^2\le1/4$ in the given range, and that $3/4\le\big(1-1/(4n)\big)^n$ for all $n\ge1$, we obtain
$$
F(\lambda)\,\le\,\prod_{j=0}^{\infty}\Big(1-\frac{a\lambda^2}{2^{2j}}\Big)^{-2^j}\le\,\prod_{j=0}^\infty\Big(\Big(1-\frac{1}{4\cdot2^{2j}}\Big)^{-2^{2j}}\Big)^{2^{-j}}
\le\,\prod_{j=0}^{\infty}(3/4)^{-2^{-j}}=\,(4/3)^{2}\,\le\, 2,
$$
as required.
\end{proof}

We may finally specify $\lambda=\frac{\beta}{4}(\log n)^k$ and apply Markov's inequality to obtain
$$
\P_{R(\rho)}\big(Z-\E_{R(\rho)}[Z]>t\big)\le F(\lambda)\,e^{-\lambda t}\le 2\exp\Big(-t\,\frac{\beta}{4}(\log n)^k\Big).
$$
For a bound on deviations below the mean, we note that
$$
-\big(Z-\E_{R(\rho)}[Z]\big)=(1-Z)-\E_{R(\rho)}[1-Z]=\P(H_{R'}^c|\eta)-\P_{R(\rho)}(H_{R'}^c).
$$
We then let $\eta'\sim\P_{R(1/\rho)}$ and note that by duality, and since cells are blue and red with equal probability, the above is equal to $\P(H_{R''}|\eta')-\P_{R(1/\rho)}(H_{R''})$ in distribution. By repeating the above proof with $Z':=\P(H_{R''}|\eta')$ in place of $Z$, via a completely analogous calculation we obtain, using~\eqref{eqn:stp2dual} in place of~\eqref{eqn:stp2}, that
$$
\P_{R(\rho)}\big(Z-\E_{R(\rho)}[Z]<-t\big)=\P_{R(1/\rho)}\big(Z'-\E_{R(1/\rho)}[Z']>t\big)\le 2\exp\Big(-t\,\frac{\beta}{4}(\log n)^k\Big).
$$
Combining the two bounds completes the proof of the proposition.
\end{proof}

\begin{remark}\label{r:N2}
Recall that $M'=M'(\eps,\beta,k)$ was defined as the least integer for which $n^{\eps/2}\ge\beta(\log n)^k$ holds for all $n\ge M'(\eps,\beta,k)$.
We shall later apply Proposition~\ref{pt2} iteratively for $\beta$ and $M$ of the form $\beta=\gamma^{k-1}$ and $M=N^{2^{k-1}}$, where $\gamma\in(0,1)$ and $N\ge1$ are constants, which $\eps$ is kept constant. For this to work, we need to make sure that for every $\eps>0$ there is an $N$ large, so that $M'(\eps,\gamma^{k-1},k)\le N^{2^{k-1}}$. This can be done in an analogous way as in Remark~\ref{r:N}.
\end{remark}

\section{Crossings of rectangles far from the boundary}\label{ImprovedResult}

In this section we take a large step towards a proof of Theorem~\ref{thm:optimized} by proving a preliminary result for rectangles $R'\subseteq R(\rho)$ contained in the bulk. By considering rectangles in the bulk, which do not align with any of the sides of the larger rectangle, we avoid effects of the tessellation that occur close to the boundary. In the next section we show how to deal with these effects, and provide a preliminary version of Theorem~\ref{thm:optimized} for arbitrary rectangles. Concluding the proof of Theorem~\ref{thm:optimized} will then be a matter of optimizing the constants, which we save for our final section.

\begin{theorem}\label{thm:bulk}
There exists $\gamma>0$ such that for every $\theta>0$ there exists $N=N(\theta)$ such that the following holds: Let $\gamma_k:=\gamma^{k-1}$ and $N_k:=N^{2^{k-1}}$. Then, for every $k\ge1$, $\rho\in[\theta,1/\theta]$, $n\ge N_k$ and $t>0$ we have
$$
\P_{R(\rho)}\Big(\big|\P(H_{R(\rho,1/4)}|\eta) - \P_{R(\rho)}(H_{R(\rho,1/4)})\big| \geq t \Big) \leq 4\exp\big(-t\,\gamma_{k}(\log n)^{k}\big).
$$
\end{theorem}

Before the proof, we recall from~\eqref{eq:rsw} that there exists $c>0$ such that
$c\le\Po(H_{R(3,n)})\le1-c$
for all $n\ge1$.
Cover the rectangle $R(3,n)$ by at most $2\sqrt{n}$ squares of area $\sqrt{n}$, and let $E_n$ denote the event that each of these squares contains a point of $\eta$. On the event $E_n$ both $\eta$ and its restriction to $R(3,2n)$ produce the same tessellation on $R(3,n)$. Since $E_n$ occurs with probability tending to one, if follows by Lemma~\ref{l:p_to_b_lower} that for some $c'>0$ and $K\ge1$ we have for $n\ge K$ that
$$
\P_{R(3,4n)}\big(H_{R(3,n)}\cap E_n\big)\ge c'\quad\text{and}\quad\P_{R(3,4n)}\big(H_{R(3,n)}^c\cap E_n\big)\ge c'.
$$
Thus, $\P_{R(3,4n)}\big(H_{R(3,n)}\big)\ge c'$ and $\P_{R(3,4n)}\big(H_{R(3,n)}^c\big)\ge c'$ for $n\ge K$, which after rescaling yields
\begin{equation}\label{eq:bincross}
c'\le\P_{R(3)}\big(H_{R(3,1/4)}\big)\le1-c'.
\end{equation}
(Alternatively, note that~\eqref{eq:bincross} also follows from Proposition~\ref{p:rsw}.)

\begin{proof}
The proof will proceed by induction, applying Propositions~\ref{pt1} and~\ref{pt2} in each step. Note that we may without loss of generality assume that $\theta\le1/3$.
We first prove the theorem for $k=1$, and then proceed with the induction step.

{\em Base step}:
By~\eqref{eq:q_rsw} we know that there exists $\delta>0$ such that for all $n\ge1$
\begin{equation}\label{eq:base}
\Po\big(\delta\le\P(H_{R(3,n)}|\eta)\le1-\delta\big)\ge1-e^{-800}.
\end{equation}
Then, by Proposition~\ref{pt1} and Remark~\ref{r:a} (applied with $k=0$, $\alpha=800$ and $L=1$), there exist $\delta'=\delta'(\delta)>0$ and $N'=N'(\theta)\ge1$ such that for all $\rho\in[\theta,1/\theta]$ and $n\ge N'$
$$
\P_{R(\rho)}\bigg(\sum_{j=1}^n\Inf_j(H_{R(\rho,1/4)}|\eta)^2>n^{-\delta'}\bigg)\le e^{-4\log n}.
$$
Since the above holds for both $\rho$ and $1/\rho$, we obtain from Proposition~\ref{pt2} (applied with $k=1$, $\beta=4$ and $M=N'$) there exists $N''=N''(\delta')$ such that for $\rho\in[\theta,1/\theta]$, $n\ge\max\{N',N''\}$ and $t>0$,
$$
\P_{R(\rho)}\Big(\big|\P(H_{R(\rho,1/4)}|\eta)-\P_{R(\rho)}(H_{R(\rho,1/4)})\big|\ge t\Big)\le 4e^{-t\log n},
$$
as required.

{\em Interlude}: In preparation for the induction step we shall fix some parameters. Let $c'>0$ and $K\ge1$ be such that~\eqref{eq:bincross} holds for $n\ge K$. Fix $\theta\le 1/3$ and let $\eps=\eps(c'/2)$ and $L''=L''(\theta)$ be as in Proposition~\ref{pt1}. Set $\gamma=c'/3200$ and let $N'''\ge1$ be the least integer (which exists, due to Remarks~\ref{r:N} and~\ref{r:N2}) such that
\begin{equation}
\text{$\gamma\log n\ge6$, $(\log n)^k\le n^{1/8}$ and $(\log n)^k\le n^{\eps/2}$ for all 
$k\ge1$ and $n\ge(N''')^{2^{k-1}}$.}
\end{equation}
Finally, set $N:=\max\{K,L'',N',N'',N'''\}$.

{\em Induction step}: Suppose the statement of the theorem is true for $k=\ell$, so that for $\rho=3$, $n\ge N_\ell=N^{2^{\ell-1}}$ and $t>0$ we have
$$
\P_{R(3)}\Big(\big|\P(H_{R(3,1/4)}|\eta)-\P_{R(3)}(H_{R(3,1/4)})\big|\ge t\Big)\le 4e^{-t\gamma_\ell(\log n)^\ell}.
$$
Taking $t=c'/2$, and recalling~\eqref{eq:bincross}, leaves us with
$$
\P_{R(3)}\Big(\frac{c'}{2}\le\P(H_{R(3,1/4)}|\eta)<1-\frac{c'}{2}\Big)\le4e^{-\frac{c'}{2}\gamma_\ell(\log n)^\ell},
$$
for $n\ge N_\ell$.
Next, Lemma~\ref{l:b_to_p} (with $R$ replaced by $R(3,4n)$ and $N$ by $n$) gives for $n\ge N_\ell$
$$
\Po\Big(\frac{c'}{2}\le\P(H_{R(3,n)}|\eta)<1-\frac{c'}{2}\Big)\le4e^{-\frac{c'}{2}\gamma_\ell(\log n)^\ell}+\Po\big(|\eta\cap R(3,4n)|<n\big)\le e^{-\frac{c'}{4}\gamma_\ell(\log n)^\ell},
$$
where we have used that $\Po\big(|\eta\cap R(3,4n)|<n\big)\le e^{-n}\le e^{-(\log n)^\ell}$ for $n\ge N_\ell$.
We are now set to apply Proposition~\ref{pt1} (for $\alpha=c'\gamma_\ell/4$, $k=\ell$ and $L=N_\ell$) to obtain that (since $N\ge \max\{L'',N'''\}$) for $n\ge N_{\ell+1}$
$$
\P_{R(\rho)}\bigg(\sum_{j=1}^n\Inf_j(H_{R(\rho,1/4)}|\eta)^2>n^{-\eps}\bigg)\le \exp\Big(-\frac{c'\gamma_\ell}{800}(\log n)^{\ell+1}\Big).
$$
Finally, applying Proposition~\ref{pt2} (for $\beta=c'\gamma_\ell/800$, $k=\ell+1$, $\eps=\eps$ and $M=N_{\ell+1}$) we obtain (since $N\ge N'''$) for all $\rho\in[\theta,1/\theta]$, $n\ge N_{\ell+1}$ and $t>0$
$$
\P_{R(\rho)}\Big(\big|\P(H_{R(\rho,1/4)}|\eta)-\P_{R(\rho)}(H_{R(\rho,1/4)})\big|\ge t\Big)\le 4\exp\Big(-t\,\frac{c'\gamma_\ell}{3200}(\log n)^{\ell+1}\Big).
$$
It follows that the theorem holds also for $k=\ell+1$, and the proof is complete.
\end{proof}

\section{Crossings of arbitrary rectangles}\label{s:boundary}

In order to prove Theorem~\ref{thm:optimized} we need to extend Theorem~\ref{thm:bulk} to include rectangles touching the boundary.



\begin{theorem}\label{thm:boundary}
There exist $\gamma>0$ such that for every $\theta>0$ there exists $N=N(\theta)$ such that the following holds: Let $\gamma_k=\gamma^{k-1}$ and $N_k=N^{2^{k-1}}$. Then, for every $k\ge 1$, $t>0$, $\rho\in [\theta,1/\theta]$ and $n\ge N_k$ we have for any (axis parallel) rectangle $R'\subseteq R(\rho)$ of area at least $\theta$ that
$$
\P_{R(\rho)}\Big(\big|\P(H_{R'}|\eta)-\P_{R(\rho)}(H_{R'})\big|\ge t\Big)\le 4e^{-t\gamma_k(\log n)^k}.
$$
\end{theorem}

Note that the aspect ratio of $R'$ lies in the interval $[\theta^2, 1/\theta^2]$.

The proof of Theorem~\ref{thm:boundary} is mostly a straightforward adaptation of the proof of Theorem~\ref{thm:bulk}, so we shall only outline the proof and highlight the distinctions. There are two main modifications required. The first is a version of Proposition~\ref{pt1} generating a bound on the sum of influences squared also when the interior rectangle is aligned with the boundary. In order to obtain this, the proposition will have to require a stronger assumption, involving bounds on the crossing probabilities of rectangles in a Voronoi tessellation of a half-plane. The second modification is that apart from~\eqref{eq:q_rsw}, we shall need also~\eqref{eq:q_rsw_halfplane}, in order to get the induction machine started.

The version of Proposition~\ref{pt1} we require is the following. Recall that $H^\ast_{R}$ denotes event of a red horizontal crossing of $R\subseteq\mathbb{H}=[0,\infty)\times\R$ in the Voronoi tessellation generated by $\eta\cap\mathbb{H}$.

\begin{proposition}\label{pt1_alt}
Suppose there exist constants $\alpha,c>0$ and $k,L\ge1$ such that for $n\ge L$ we have
\begin{align}
\label{cond:halfplane1}
\Po\Big(c<\P(H_{R(3,n)}|\eta)<1-c\Big)&\ge1-e^{-\alpha(\log n)^k},\\
\label{cond:halfplane2}
\Po\Big(c<\P(H^\ast_{R_0(3/2,n)}|\eta)<1-c\Big)&\ge1-e^{-\alpha(\log n)^k}.
\end{align}
Then, for every $\theta>0$ there exist constants $\eps=\eps(c)$, $L'=L'(\alpha,k)$ and $L''=L''(\theta)$ such that for all $\rho\in[\theta,1/\theta]$, $n\ge\max\{L^2,L',L''\}$ and (axis parallel) rectangle $R'\subseteq R(\rho)$ of area at least $\theta$, we have
$$
\P_{R(\rho)}\bigg(\sum_{j=1}^n\Inf_j(H_{R'}|\eta)^2\ge n^{-\eps}\bigg)\le e^{-\frac{\alpha}{200}(\log n)^{k+1}}.
$$
\end{proposition}

As with Proposition~\ref{pt1}, the proof will consist of an estimate on the one-arm event in combination with an application of the Schramm-Steif Revealment Theorem. It is for the former of the two that we require the additional assumption, so that we can bound the arm event not only in the bulk, but also for points close to the boundary.

\begin{lemma}\label{OneArmImp_alt}
Under the assumptions of Proposition~\ref{pt1_alt}, there exist 
$\eps=\eps(c)$, $L'(\alpha,k)$ and $L''(\theta)$ such that for every $\rho\in[\theta,1/\theta]$ and $u\in R(\rho,n)$ we have for $n\ge\max\{L^2,L'(\alpha,k),L''(\theta)\}$ that
$$
\P_{R(\rho,n)}\Big(\P\big(V_u(n^{1/4},n^{1/3})|\eta\big)\ge n^{-\eps}\Big)\le e^{-\frac{\alpha}{100}(\log n)^{k+1}}.
$$
\end{lemma}

\begin{proof}
The proof is analogous to the proof of Lemma~\ref{OneArmImp}, with minor modifications for the case when $u$ is close to the boundary. In fact, we shall distinguish between three cases depending on whether $u$ is close to a \emph{corner}, close to a \emph{side}, or within the \emph{bulk} of the rectangle $R(\rho,n)$. We define being close to a corner as being within (Euclidean) distance $n^{19/60}$ of two of the sides; being close to a side as being within distance $n^{17/60}$ of one of the sides, but at distance at least $n^{19/60}$ to the remaining sides; and being in the bulk as being at distance at least $n^{17/60}$ to all sides of $R(\rho,n)$. (It is convenient to think of the interval $[1/4,1/3]$ split into five intervals with endpoints $1/4=15/60$, $16/60$, $17/60$, $18/60$, $19/60$ and $1/3=20/60$.)

{\em Case 1: corner.} Let $x$ denote one of the corners of $R(\rho,n)$ and consider $u$ within Euclidean distance $n^{19/60}$ of the two sides associated with the corner, so that $u$ is within $\ell_\infty$-distance $n^{19/60}$ of $x$. Consider the annuli $A_j=x+[-\frac327^j,\frac327^j]\setminus[-\frac127^j,\frac127^j]$ for $j$ in $J':=\{j\in\N:4n^{19/60}\le 7^j\le\frac12n^{1/3}\}$. Only a quarter of these annuli are contained within $R(\rho,n)$. Hence, the absence of a red crossing from the inner to the outer boundary of $A_j$, which we again denote by $O_j$, will in this case corresponds to a blue path connecting two sides of $R(\rho,n)$ within $A_j$. The quarter of $A_j$ contained within $R(\rho,n)$ can be covered by two rectangles with aspect ratio $\frac32:1$, in such a way that a crossing of each of these rectangles imply said connection between the two sides. Thus, Harris' inequality and~\eqref{cond:halfplane2} give
$$
\Po\big(\P(O_j|\eta)<c^2\big)<2e^{-\alpha(\log n)^k},
$$
for all $n\ge L^2$. Proceeding as before provides a bound on the conditional probability of $V_u(n^{19/60},n^{1/3})$ of the correct order, valid uniformly over all corner points.

{\em Case 2: side.} Next, let $u\in R(\rho,n)$ be any point within Euclidean distance $n^{17/60}$ to the boundary of $R(\rho,n)$, but at $\ell_\infty$-distance at least $n^{19/60}$ to a corner, and let $y$ denote the boundary point closest to $u$. Consider  $A_j=y+[-\frac327^j,\frac327^j]\setminus[-\frac127^j,\frac127^j]$ for $j$ in $J'':=\{j\in\N:4n^{17/60}\le7^j\le\frac12n^{18/60}\}$. Note that only half of $A_j$ is contained inside $R(\rho,n)$, and that this half can be covered by three rectangles, two of ratio $\frac32:1$ and one of ratio $3:1$, such that a blue crossing of each of these impedes a red crossing from the inner to the outer boundary of $A_j$; the event denoted by $O_j$. Again, Harris' inequality, combined with both~\eqref{cond:halfplane1} and~\eqref{cond:halfplane2}, gives
$$
\Po\big(\P(O_j|\eta)<c^3\big)<3e^{-\alpha(\log n)^k},
$$
for all $n\ge L^2$. Proceeding as before results in a bound on the conditional probability of the event $V_u(n^{17/60},n^{18/60})$.

{\em Case 3: bulk.} For points in the bulk of the rectangle straightforward adaptations of the proof of Lemma~\ref{OneArmImp} will provide a bound on the conditional probability of $V_u(n^{1/4},n^{16/60})$ of the correct order.

The three different cases described above require minimal modification from one another (which in each case may give rise to slightly different values of the involved constants). Together, they complete the proof of the lemma.
\end{proof}

\begin{proof}[Proof of Proposition~\ref{pt1_alt}]
The proof is a straightforward adaptation of the proof of Proposition~\ref{pt1}, using $R'$ instead of $R(\rho,1/4)$ and Lemma~\ref{OneArmImp_alt} instead of Lemma~\ref{OneArmImp}.
\end{proof}

\begin{proof}[Proof of Theorem~\ref{thm:boundary}]
The proof will proceed by induction, much like the proof of Theorem~\ref{thm:bulk}. Note that we may without loss of generality assume that $\theta\le1/3$.
We first prove the theorem for $k=1$, and then proceed with the induction step.

{\em Base step}:
By~\eqref{eq:q_rsw_halfplane} we know that there exists $\delta>0$ such that for all $n\ge1$
$$
\Po\big(\delta\le\P(H^\ast_{R_0(3/2,n)}|\eta)\le1-\delta\big)\ge1-e^{-800}.
$$
Together with~\eqref{eq:base}, Proposition~\ref{pt1_alt} and Remark~\ref{r:a} (applied with $k=0$, $\alpha=800$ and $L=1$), there exist $\delta'=\delta'(\delta)>0$ and $N'=N'(\theta)\ge1$ such that for all $\rho\in[\theta,1/\theta]$, $n\ge N'$ and any rectangle $R'\subseteq R(\rho)$ of area at least $\theta$
\begin{equation}\label{eq:infsum}
\P_{R(\rho)}\bigg(\sum_{j=1}^n\Inf_j(H_{R'}|\eta)^2>n^{-\delta'}\bigg)\le e^{-4\log n}.
\end{equation}
Since rotation of $R(\rho)$ by $\pi/2$ results in $R(1/\rho)$, we obtain from Proposition~\ref{pt2} (applied with $k=1$, $\beta=4$ and $M=N'$) there exists $N''=N''(\delta')$ such that for $\rho\in[\theta,1/\theta]$, $n\ge\max\{N',N''\}$ and $t>0$,
$$
\P_{R(\rho)}\Big(\big|\P(H_{R'}|\eta)-\P_{R(\rho)}(H_{R'})\big|\ge t\Big)\le 4e^{-t\log n},
$$
as required.

{\em Interlude}: In preparation for the induction step we shall fix some parameters. Fix $\theta\le1/3$. Let $c''>0$ and $K\ge1$ be such that~\eqref{eq:bincross} holds for $n\ge K$. Let $\eps=\eps(c''/2)$ and $L''=L''(\theta)$ be as in Proposition~\ref{pt1_alt}. Set $\gamma=c''/3200$ and let $N'''\ge1$ be the least integer (which exists, due to Remarks~\ref{r:N} and~\ref{r:N2}) such that
\begin{equation}
\text{$\gamma\log n\ge6$, $(\log n)^k\le n^{1/8}$ and $(\log n)^k\le n^{\eps/2}$ for all 
$k\ge1$ and $n\ge(N''')^{2^{k-1}}$.}
\end{equation}
Finally, let $N:=\max\{K,L'',N',N'',N'''\}$.

{\em Induction step}: Suppose the statement of the theorem is true for $k=\ell$. Setting $R'=R(3,1/4)$ and $t=c''/2$ we obtain, just as in the proof of Theorem~\ref{thm:bulk}, that for $n\ge N_\ell$
$$
\Po\Big(\frac{c''}{2}\le\P(H_{R(3,n)}|\eta)<1-\frac{c''}{2}\Big)\geq 1-e^{-\frac{c''}{4}\gamma_\ell(\log n)^\ell}.
$$
Instead, setting $R'=[-\sqrt{3}/2,-\sqrt{3}/4]\times[-1/(4\sqrt{3}),1/(4\sqrt{3})]$, so that $R'$ is a $\frac32:1$ rectangle aligned with the left side of $R(3)$, leads in an analogous manner to the bound, for $n\ge N_\ell$,
$$
\Po\Big(\frac{c''}{2}\le\P(H^\ast_{R_0(3/2,n)}|\eta)<1-\frac{c''}{2}\Big)\geq 1-e^{-\frac{c''}{4}\gamma_\ell(\log n)^\ell},
$$
where $c''$ is the constant of Proposition~\ref{p:rsw}.  And indeed, this argument relies on Proposition~\ref{p:rsw}, which gives a bound on the annealed crossing probability.

Now, Proposition~\ref{pt1_alt} (for $\alpha=c''\gamma_\ell/4$, $k=\ell$ and $L=N_\ell$) gives (since $N\ge \max\{L'',N'''\}$) for $n\ge N_{\ell+1}$ and any axis parallel rectangle $R'\subseteq R(\rho)$ of area at least $\theta$ that
$$
\P_{R(\rho)}\bigg(\sum_{j=1}^n\Inf_j(H_{R'}|\eta)^2>n^{-\eps}\bigg)\le \exp\Big(-\frac{c''\gamma_\ell}{800}(\log n)^{\ell+1}\Big).
$$
Finally, applying Proposition~\ref{pt2} (for $\beta=c''\gamma_\ell/800$, $k=\ell+1$ and $M=N_{\ell+1}$) we obtain (since $N\ge N'''$) for all $\rho\in[\theta,1/\theta]$, $n\ge N_{\ell+1}$, $t>0$ and any axis parallel rectangle $R'\subseteq R(\rho)$ of area at least $\theta$ that
$$
\P_{R(\rho)}\Big(\big|\P(H_{R'}|\eta)-\P_{R(\rho)}(H_{R'})\big|\ge t\Big)\le 4\exp\Big(-t\,\frac{c''\gamma_\ell}{3200}(\log n)^{\ell+1}\Big).
$$
Hence, the theorem holds also for $k=\ell+1$, and the proof is complete.
\end{proof}

\section{Optimizing the rate of decay}\label{s:optimized}

Completing the proof of Theorem~\ref{thm:optimized} is now a matter of verification.

\begin{proof}[Proof of Theorem~\ref{thm:optimized}]
Let $\gamma>0$ and $N=N(\theta)$ be as in Theorem~\ref{thm:boundary}. For $n\ge1$, let $k(n)$ denote the integer such that
$$
N^{2^{k(n)-1}}<n\le N^{2^{k(n)}}.
$$
The upper inequality implies that
$$
k(n)\ge\frac{\log\log n-\log\log N}{\log 2}\ge\frac43\log\log n,
$$
for $n$ larger than some $N'\ge N$. Theorem~\ref{thm:boundary} thus gives that
\begin{align*}
    \P_{R(\rho)}\big(|\P(H_{R'}|\eta)-\P_{R(\rho)}(H_{R'})|\ge t\big)&\le4\exp\Big(-t\gamma_{k(n)}(\log n)^{k(n)}\Big)\\
    &\le4\exp\Big(-t\gamma_{k(n)}(2^{k(n)-1}\log N)^m(\log n)^{k(n)-m}\Big),
\end{align*}
for any integer $m$, and $n\ge N'$. Taking $m$ large, so that $\gamma 2^m>1$, leads to the further upper bound
$$
4\exp\Big(-t(\log n)^{k(n)-m}\Big)\le4\exp\Big(-t(\log n)^{\log\log n}\Big)=4\exp\Big(-t\,e^{(\log\log n)^2}\Big)
$$
for all $n$ sufficiently large.
\end{proof}

In the next section we shall prove Theorem~\ref{thm:poptimized}, which is an analogue of Theorem~\ref{thm:optimized} for the Poisson model. However, let us illustrate already here that a weaker statement follows easily from Theorem~\ref{thm:optimized} through comparison.

\begin{corollary}\label{cor:Poisson}
For every $\theta>0$ there exists $c=c(\theta)>0$ and $N=N(\theta)\ge1$ such that for all $\rho\in[\theta,1/\theta]$ and $n\ge N$ we have
$$
\Po\big(c<\P(H_{R(\rho,n)}|\eta)<1-c\big)\ge1-\exp\big(-e^{(\log\log n)^2}\big).
$$
\end{corollary}

\begin{proof}
Let $c'$ be as in~\eqref{eq:bincross} and let $\rho\in[\theta,1/\theta]$. From (the proof of) Theorem~\ref{thm:optimized} for sufficiently large $n$ (depending on $\theta$) we have
$$
\P_{R(\rho)}\Big(\frac{c'}{2}\le\P(H_{R(\rho,1/4)}|\eta)\le1-\frac{c'}{2}\Big)\ge1-4\exp\big(-\frac{c'}{2}e^{\frac76(\log\log n)^2}\big)\ge1-\frac12\exp\big(-e^{(\log\log n)^2}\big).
$$
Lemma~\ref{l:b_to_p} thus gives, for such values of $n$, that
\begin{align*}
\Po\Big(\frac{c'}{2}<\P(H_{R(\rho,n)}|\eta)<1-\frac{c'}{2}\Big)&\ge1-\frac12\exp\big(-e^{(\log\log n)^2}\big)-\Po\big(|\eta\cap R(\rho,4n)|<n\big)\\
&\ge1-\frac12\exp\big(-e^{(\log\log n)^2}\big)-e^{-n}\\
&\ge1-\exp\big(-e^{(\log\log n)^2}\big),
\end{align*}
as required.
\end{proof}

\section{Rate in the Poisson model}\label{sec:poptimized}

In this section we prove Theorem \ref{thm:poptimized}, which provides a version of Theorem~\ref{thm:optimized} in the Poisson setting. We saw already in the previous section how to derive a weaker statement through comparison. In order to obtain the full statement of Theorem~\ref{thm:poptimized}, we shall need to complement the comparison lemma with an argument showing that the probability of crossing a given rectangle changes very little when adding a small number of points.

Throughout the paper we have used $n$ to denote the number of points when the point set $\eta$ is chosen in the binomial model, but let this number be implicit in the notation. In this section we shall need to compare probabilities for crossing a fixed rectangle for different numbers of points. We shall therefore need to make the number of points considered explicit in our notation, and write $\P_{R(\rho,n),m}$ for the probability measure by which $\eta$ is selected as a set of $m$ uniformly random points in the rectangle $R(\rho,n)$. As before we write $R(\rho)$ for $R(\rho,1)$.

The first step will be a lemma which states that the crossing probability in this model does not change quickly as $m$ changes, which may be of some interest on its own.

\begin{lemma}\label{lem:Rrnm} There exists a constant $c>0$ such that for every $\theta>0$ there exists $N_0=N_0(\theta)$ such that the following holds for all $\rho\in [\theta,1/\theta]$ and $n\ge N_0$: For every $m$ satisfying $|m-n|<n^{1/2+c}$ and every (axis parallel) rectangle $R'\subseteq R(\rho)$ with area at least $\theta$ we have
$$
\big|\P_{R(\rho),m}(H_{R'})\, -\, \P_{R(\rho),n}(H_{R'})\big|\, <\, n^{-c}.
$$
\end{lemma}

\begin{remark} This crossing probability is believed to be stable over much larger intervals.  In particular, if Voronoi percolation exhibits conformal invariance, then the crossing probability converges as the number of point increases.
\end{remark}

\begin{proof} We shall prove that adding a single point has very little effect.  Specifically, we shall prove that there exists a constant $c>0$ such that for $n$ large and any $m$ with $|m-n|< n^{1/2+c}$ we have
\begin{equation}\label{eq:onebound}
\big|\P_{R(\rho),m}(H_{R'})\, -\, \P_{R(\rho),m-1}(H_{R'})\big|\, <\, n^{-(1/2+2c)}.
\end{equation}
The lemma then follows immediately using the triangle inequality.

Let $\eta$ denotes a set of $m$ uniformly random points in $R(\rho)$, and let $\eta_-$ be obtained by removing one point uniformly at random.  We then have
\begin{align*}
\big|\P_{R(\rho),m}(H_{R'})\, -\, \P_{R(\rho),m-1}(H_{R'})\big|\, & =\, \big|\E_{R(\rho),m}\big[\P(H_{R'}|\eta)-\P(H_{R'}\big|\eta_{-})\big]\big|\\ 
& \le\, \E_{R(\rho),m}\left[\big|\P(H_{R'}|\eta)-\P(H_{R'}|\eta_{-})\big|\right]\\
& \le\, \E_{R(\rho),m}\bigg[\frac{1}{m}\sum_{i=1}^{m}\Inf_{i}(H_{R'}|\eta)\bigg]\, ,
\end{align*}
where the final inequality follows from~\eqref{eq:exclude} and the fact that $\eta_{-}$ is obtained by removing a randomly chosen point.

Note further that
\[
\frac{1}{m}\sum_{i=1}^{m}\Inf_{i}(H_{R'}|\eta)\, \le\, \bigg(\frac{1}{m}\sum_{i=1}^{m}\Inf_{i}(H_{R'}|\eta)^2\bigg)^{1/2}
\]
by the Cauchy-Schwarz inequality. Recall that for monotone events the sum of influences squared are bounded by 1. (This follows e.g.\ by~\eqref{eq:SS}.) Moreover, by~\eqref{eq:infsum} there exists $c'>0$ such that the sum of influences squared is at most $m^{-c'}$ with probability at least $1-m^{-c'}$, when $m$ is large. It follows, using Jensen's inequality, that for large values of $m$
$$
\big|\P_{R(\rho),m}(H_{R'})\, -\, \P_{R(\rho),m-1}(H_{R'})\big|\,\le\, \E_{R(\rho),m}\bigg[\frac1m\sum_{i=1}^{m}\Inf_{i}(H_{R'}|\eta)^2\bigg]^{1/2} \le\, 2m^{-(1/2+c'/2)}.
$$
Setting $c=c'/8$ this proves~\eqref{eq:onebound} in the given range on $m$ and $n$, as required.
\end{proof}

We shall deduce from Lemma~\ref{lem:Rrnm} that the probabilities of crossing a large rectangle is the two models are close when the actual number of points in the binomial model is close to the expected number of points in the Poisson model. More precisely, we shall establish that for some $c>0$ we have
\begin{equation}\label{Pom}
\big|\P_{R(\rho,2n),m}(H_{R(\rho,n)})-\Po(H_{R(\rho,n)})\big|\, <\, 6n^{-c}
\end{equation}
whenever $\rho\in[\theta,1/\theta]$ and $|m-2n|<n^{1/2+c}$, provided that $n$ is sufficiently large.

To this end, let $c>0$ and $N_0(\theta)$ be as in Lemma~\ref{lem:Rrnm} and let $\theta>0$ be fixed. We then have
\[
\big|\P_{R(\rho,2n),m}(H_{R(\rho,n)})-\P_{R(\rho,2n),2n}(H_{R(\rho,n)})\big|\, <\, n^{-c}
\]
whenever $\rho\in[\theta,1/\theta]$, $|m-2n|<n^{1/2+c}$ and $n\ge N_0(\theta)$.

Cover $R(\rho,2n)$ by squares of area $\rho n/1000$, and let $F$ be the event that every square contains a point of $\eta$. An argument similar to the one that appears in the proof of Claim~\ref{claim:revealment} shows that the probability of $F^{c}$ is exponentially small in $n$ with respect to both $\Po$ and $\P_{R(\rho,2n),m}$ when $m\ge n$. And so, writing $E$ for the event $H_{R(\rho,n)}\cap F$, we have both $|\Po(H_{R(\rho,n)})-\Po(E)|<n^{-c}$ and
\begin{equation}\label{eq:varyE}
\big|\P_{R(\rho,2n),m}(E)-\P_{R(\rho,2n),2n}(E)\big|\, <\, 2n^{-c}
\end{equation}
whenever $|m-2n|<n^{1/2+c}$ and $n$ is large. As in Section~\ref{Comparison}, we may express the Poisson probability by conditioning on the number of points of $\eta$ in the rectangle.  Let $A_m$ be the event that $|\eta\cap R(\rho,2n)|=m$.  Since $E$ is measurable with respect to points and colours inside $R(\rho,2n)$ we have
\begin{align*}
\Po(E)\, & =\, \sum_{m\ge 0} \P(E|A_m)\Po(A_m)\\
& =\, \sum_{m\ge 0}\P_{R(\rho,2n),m}(E)\Po(A_m)\, .
\end{align*}
Since $\Po$ puts mass at most $\exp(-n^{c})$ that the (Poisson distributed) number of points falls outside the interval $2n\pm n^{1/2+c}$, we have from~\eqref{eq:varyE} (applied twice) that
\[
\big|\P_{R(\rho,2n),m}(E)-\Po(E)\big|\, <\, \big|\P_{R(\rho,2n),2n}(E)-\Po(E)\big|+2n^{-c}\,<\, 5n^{-c} 
\]
whenever $|m-2n|<n^{1/2+c}$ and $n$ is large, and so~\eqref{Pom} follows.

\begin{proof}[Proof of Theorem~\ref{thm:poptimized}]
Fix $\theta>0$ and let $p_n(t):=\exp\big(-t\,e^{(\log\log n)^2}\big)$. Without loss of generality, we may in the following suppose that $t\le1$. Consequently, $\Po(F^c)\le p_n(t)$, where $F$ is the event defined earlier in this section, and $\Po$ puts mass at most $p_n(t)$ that the number of points in $R(\rho,2n)$ falls outside the interval $2n\pm n^{1/2+c}$, for all sufficiently large $n$.

We have
\begin{align*}
\Po\big(&|\P(H_{R(\rho,n)}|\eta)-\Po(H_{R(\rho,n)})|>t\big)\,  \le \, \Po\big(\{|\P(H_{R(\rho,n)}|\eta)-\Po(H_{R(\rho,n)})|>t\}\cap F\big)\, +\, \Po(F^{c})\phantom{\Big|}\\
&\le\, \sum_{m\ge 0} \Po\big(\{|\P(H_{R(\rho,n)}|\eta)-\Po(H_{R(\rho,n)})|>t\}\cap F\big|A_m\big)\Po(A_m)\, +\, p_n(t)\phantom{\Big|}\\
&\le\, \max\big\{ \P_{R(\rho,2n),m}\big(\{|\P(H_{R
(\rho,n)}|\eta)-\Po(H_{R(\rho,n)})|>t\}\cap F\big)\, :\, |m-2n|\le n^{1/2+c}\big\}\, +\, 2p_n(t).
\end{align*}
For $t<12n^{-c}$ the upper bound in the theorem is trivial, since $p_n(t)>1$ for large $n$, so there is nothing to prove. For $t\ge12n^{-c}$ we obtain from~\eqref{Pom} the further upper bound
$$
\max\big\{ \P_{R(\rho,2n),m}\big(\{|\P(H_{R(\rho,n)}|\eta)-\P_{R(\rho,2n),m}(H_{R(\rho,n)})|>t/2\}\cap F\big)\, :\, m\ge n\big\}\, +\, 2p_n(t),
$$
which by Theorem~\ref{thm:optimized} is no larger than $6p_n(t)$.
\end{proof}

\appendix

\section{The revealment theorem}

Here, in this appendix, we shall provide a short elementary and probabilistic proof of the version of the Schramm-Steif revealment theorem as presented in~\eqref{eq:SS}. The proof is similar to that of an inequality due to O'Donnell and Servedio~\cite{odonnell2007}. For notational convenience during the proof we phrase the result in terms of subsets of the cube $\{-1,1\}^n$.

\begin{proposition}
Let $A\subseteq\{-1,1\}^n$ be a monotone event and let $\mathcal{A}$ be a (randomized) algorithm that determines $A$. Then, for any subset $J\subseteq[n]$ we have
$$
\sum_{j\in J}\Inf_j(A)^2\le\max_{j\in J}\P(\mathcal{A}\text{ queries }j).
$$
\end{proposition}

Note that the proposition gives a bound on the sum of influences squared over a subset $J$ of the bits. It thus extends and simplifies a precursor to the Schramm-Steif Revealment Theorem which is due to Benjamini-Kalai-Schramm; see~\cite[Theorem~12.52]{garban_steif_book} for a precise statement.

\begin{proof}
Let $f:\{-1,1\}^n\to\{-1,1\}$ be the function that takes the value $1$ for $\omega\in A$ and $-1$ for $\omega\not\in A$. Then $f$ is monotone (increasing) and we have $\Inf_j(A)=\E[f\omega_j]$ for each $j=1,2,\ldots,n$. We thus obtain
$$
\sum_{j\in J}\Inf_j(f)^2=\sum_{j\in J}\Inf_j(f)\E[f\omega_j]=\E\bigg[f\,\sum_{j\in J}\Inf_j(f)\omega_j\bigg]=\E\bigg[ f\,\E\bigg[\sum_{j\in J}\Inf_j(f)\omega_j\bigg|\mathcal{F}\bigg]\bigg],
$$
where we have written $\mathcal{F}$ for the information revealed by the algorithm. Applying Cauchy-Schwartz gives us that
$$
\sum_{j\in J}\Inf_j(f)^2\le\bigg(\E[f^2]\,\E\bigg[\E\bigg[\sum_{j\in J}\Inf_j(f)\omega_j\bigg|\mathcal{F}\bigg]^2\bigg]\bigg)^{1/2}.
$$
Since there is no information regarding the state of the bits not queried by the algorithm in $\mathcal{F}$, the square of the right-hand side in the above expression can be rewritten as
$$
\E\bigg[\bigg(\sum_{j\in J}\Inf_j(f)\omega_j\Ind_{\{\mathcal{A}\text{ queries }j\}}\bigg)^2\bigg]=\E\bigg[\sum_{i,j\in J}\Inf_i(f)\Inf_j(f)\omega_i\omega_j\Ind_{\{\mathcal{A}\text{ queries $i$ and }j\}}\bigg].
$$
When the algorithm first queries one of $i$ and $j$ it has no information regarding the other. Consequently, the `off diagonal' entries of the double sum will be zero, so that
$$
\bigg(\sum_{j\in J}\Inf_j(f)^2\bigg)^2\le\sum_{j\in J}\Inf_j(f)^2\,\P(\mathcal{A}\text{ queries }j)\le\max_{i\in J}\P(\mathcal{A}\text{ queries }i)\sum_{j\in J}\Inf_j(f)^2.
$$
Rearranging the two sides gives the claimed result.
\end{proof}

\printbibliography

{\small
\noindent
{\sc 

Department of Mathematics, Stockholm University\\ 
SE-10691 Stockholm, Sweden}\\

\noindent
{\sc 
Departamento de Matem\'atica, PUC-Rio \\
Rua Marqu\^{e}s de S\~{a}o Vicente 225, G\'avea, 22451-900 Rio de Janeiro, Brasil}\\

\end{document}